\documentclass[11pt,reqno]{amsart}
\usepackage{amsmath}
\usepackage{amsthm}
\usepackage{graphicx}
\usepackage{hyperref}
\usepackage{amsfonts}
\usepackage{amssymb}
\usepackage{enumerate}
\usepackage[margin=1.1in]{geometry}
\usepackage[foot]{amsaddr}

\usepackage[english]{babel}
\usepackage[latin1]{inputenc}
\usepackage[numbers]{natbib}

\usepackage[normalem]{ulem}

\DeclareOldFontCommand{\bf}{\normalfont\bfseries}{\mathbf}
\DeclareOldFontCommand{\cal}{\normalfont\bfseries}{\mathcal}

\newtheorem{theorem}{Theorem}[section]
\newtheorem{lemma}[theorem]{Lemma}
\newtheorem{proposition}[theorem]{Proposition}
\newtheorem{corollary}[theorem]{Corollary}

\theoremstyle{definition}
\newtheorem{example}[theorem]{Example}

\newtheorem{remark}[theorem]{Remark}

\def\R{{\mathbb R}}

\def\N{{\mathbb N}}
\def\PP{{\mathbb P}}
\def\FF{{\mathbb F}}
\def\P{{\mathcal P}}

\def\Q{{\mathcal Q}}
\def\L{{\mathcal L}}

\def\A{{\mathcal A}}

\def\F{{\mathcal F}}
\def\Rel{{\mathcal R}}

\DeclareMathOperator*{\esssup}{ess\,sup}

\def\W{P}
\def\E{{\mathbb E}}
\def\C{{\mathcal C}}

\pdfoutput=1
\usepackage{color}

\begin{document}

\title[Non-exponential Sanov and Schilder theorems on Wiener space]{Non-exponential Sanov and Schilder theorems on Wiener space: BSDEs, Schr\"odinger problems and Control.}

\author{Julio Backhoff-Veraguas}
\address{University of Vienna, Stochastics and Financial Mathematics Group}
\author{Daniel Lacker}
\address{Columbia University, Industrial Engineering and Operations Research}
\author{Ludovic Tangpi}
\address{Princeton University, Operations Research and Financial Engineering}

\maketitle

\begin{center}
\today
\end{center}

\begin{abstract}
We derive new limit theorems for Brownian motion, which can be seen as non-exponential analogues of the large deviation theorems of Sanov and Schilder in their Laplace principle forms. As a first application, we obtain novel scaling limits of backward stochastic differential equations and their related partial differential equations. As a second application, we extend prior results on the small-noise limit of  the Schr\"odinger problem as an optimal transport cost, unifying the control-theoretic and probabilistic approaches initiated respectively by T.\ Mikami and C.\ L\'eonard. Lastly, our results suggest a new scheme for the computation of mean field optimal control problems, distinct from the conventional particle approximation. A key ingredient in our analysis is an extension of the classical variational formula (often attributed to Borell or Bou\'e-Dupuis) for the Laplace transform of Wiener measure.
\end{abstract}

\tableofcontents

\section{Introduction}

In this work we develop two new limit theorems for the Wiener process {along with several applications}. 
These can be seen as non-exponential extensions of the classical large deviation principles of Schilder and Sanov in their Laplace principle forms. Our findings build on the recent limit theorems obtained in the article \cite{Lac-Sanov} by the second named author in an abstract setting.  Along the way, we derive a variational principle for the Wiener process which can be seen as a reformulation of Gibbs variational principle as initiated by \cite{Fl78,Boue-Dup}; see also \cite{Bo00,Lehec} for further developments. Our two limit theorems turn out to be a common ground for three domains of application, as we now describe.

Our first application concerns the theory of backward stochastic differential equations (BSDE), and their related convex dual and PDE representations.
Our two main limit theorems lead to two new kinds of scaling limits for BSDEs. One of these scaling limits can be seen as a non-Markovian vanishing-viscosity limit. Indeed, by exploiting the well-known link between BSDEs and semilinear PDEs (see e.g.\ \cite{Pardoux-Peng92,Pardoux-Tang99}), our result recovers as a special case the well-known convergence of a viscous Hamilton-Jacobi equation to its inviscid counterpart as the viscosity coefficient vanishes.
In our second and more peculiar BSDE scaling limit, the terminal condition depends on the empirical distribution of $n$ rescaled sub-paths of the Brownian motion; although decidedly non-Markovian, in a special case this translates to a limit theorem for ``concatenated'' semilinear PDEs.

Our second application concerns the convergence of Schr\"odinger-type problems (also called stochastic optimal transport) to classical optimal transport in the small noise limit. 
The Schr\"odinger problem is a classical topic in probability theory and mechanics, see e.g.\ \cite{Le14} and the references therein, and its link to optimal transportation was developed by \citet{Foel88} in his Saint Flour lecture notes.
The study of small-noise limits of Schr\"odinger problems was pioneered by Mikami in the works \cite{Mi04,MiTh08}, the second joint with Thieullen. The main tool in these articles was stochastic control and partial differential equations (PDEs). Subsequently, an elegant large deviations viewpoint was  developed by L\'eonard in \cite{Le12,Le16}. We draw inspiration from both approaches, to a certain extend unifying them, as we exploit our limit theorems in order to obtain new small-noise results for Schr\"odinger-type problems.

The third application is a surprising connection with a particular type of optimal control problem, known as \emph{mean field} or \emph{McKean-Vlasov optimal control}, which have seen a surge of interest in recent years; see \cite{CarmonaDelarue-MVcontrol,PhamWei-MVcontrol,Lacker-MVcontrol} and references therein. The limiting quantity in our Sanov-type theorems can be seen as the value of an optimal control problem in which the dependence of the optimization criterion on the law of the state process is nonlinear. Our limit theorem provides a peculiar new approximation scheme for such problems, markedly different from the natural particle approximation worked out in \cite{Lacker-MVcontrol}.

So far we have superficially described the contributions of this article. We now proceed to present the setting and main results in detail.

\section{Setting and main results}
\label{sec setting}

Let $\C=C([0,1];\R^d)$ denote the continuous path space, equipped with the supremum norm $\|\cdot\|_\infty$, and its Borel $\sigma$-field. Let $\W$ denote the standard Wiener measure on $\C$. With $W=(W(t))_{t \in [0,1]}$ we denote the canonical (coordinate) process on $\C$, defined by setting $W(t)(\omega)=\omega(t)$, so that $W$ is a standard $d$-dimensional Brownian motion under $\W$. Let $\FF=(\F_t)_{t \in [0,1]}$ denote the $\W$-complete filtration generated by $W$.
As usual, we denote by $L^0(\W)$ the space of (real-valued) random variables quotiented with the $\W$-a.s.\  identification, and by $L^\infty(\W)$ the essentially bounded elements of $L^0(\W)$. We will likewise identify processes that are $dt\otimes d\W$-almost surely equal. 

Throughout the paper we consider $g : [0,1] \times \R^d \rightarrow \R \cup\{\infty\}$, a function whose effective domain we define as $\mathrm{dom}(g(t,\cdot)) := \{q \in \R^d : g(t,q) < \infty \}$, and satisfying the assumption:
\begin{itemize}
\item[{(TI)}] 
The function $g$ is measurable and bounded from below, and it is coercive in the sense that $\lim_{|q|\to \infty}\inf_{t \in [0,1]}\frac{g(t,q)}{|q|} = \infty$. For each $t \in [0,1]$ the function $g(t,\cdot)$ is convex, proper, and lower semicontinuous. {Finally, the following technical conditions hold}:
\begin{align}
\label{TI:domain}
	0 \in \mathrm{ri}(\mathrm{dom}(g(t,\cdot))) = \mathrm{ri}(\mathrm{dom}(g(s,\cdot)))=:\Rel\quad \text{for all } s,t \in [0,1]
\end{align}
and
\begin{align}
	\sup_{\substack{|q| \le r\\ q \in \Rel }}g(t,q) \in L^1([0,1],dt)
	\label{TI: time integrability} \quad \text{for all } r\ge 0,
\end{align}
where $\mathrm{ri}(\mathrm{dom}(g(t,\cdot)))$ denotes the relative interior of $\mathrm{dom}(g(t,\cdot))$. 
\end{itemize}
The final technical conditions \eqref{TI:domain} and \eqref{TI: time integrability} always holds if $g$ is finite-valued and jointly continuous. {A typical example which takes the value $+\infty$ and which satisfies (TI) is $g(t,\cdot) = +\infty 1_{K}(q)$, the convex indicator of a convex compact set $K \subset \R^d$.}
{The assumption that $0 \in \Rel$ is unnecessary, but it is convenient and not terribly restrictive.}

Define $\L$  to be the set of progressively measurable $\R^d$-valued processes $q : [0,1] \times \C \rightarrow \R^d$ satisfying $\W(\int_{0}^{1}|q(t)|^2dt < \infty) = 1$.  We often write $q(t)=q(t,\cdot)$, suppressing the dependence on $\omega \in \C$. We denote by $\int_0^1q^Q(t)\,dW(t)$ the stochastic integral $\int_0^1q^Q(t)\cdot dW(t)$. Let $\L_b \subset \L$ denote the subset of \emph{bounded} processes.
Let ${\cal Q}$ be the set of probability measures absolutely continuous with respect to $\W$.
It is well known that for every $Q \in {\cal Q}$, there is a unique process  $q^Q \in {\cal L}$ such that $Q$-a.s.
\begin{equation*}
	\frac{dQ}{d\W}= \exp\left(\int_0^1q^Q(t)\,dW(t)-\int_0^1\frac{1}{2}|q^Q(t)|^2\,dt \right).
\end{equation*}
{ A partial converse which we will often use is as follows: for any $q \in \L_b$, there is a unique $Q^q \in \Q$ such that $q^{Q^q}=q$. By Girsanov's theorem, we may express this measure as} {
\begin{align}
Q^q = \W \circ \left(W + \int_0^\cdot q(t)dt\right)^{-1}. \label{def:Q^q}
\end{align}}

The main objects we study are the conjugate functionals
\begin{align*}
\alpha^g :{\cal Q}\to \R\cup\{+\infty\}, \quad\quad \rho^g:L^\infty(\W)\to \R\cup\{+\infty\},
\end{align*}
respectively given by
\begin{equation}\label{eq def alpha rho}
	\alpha^g(Q):= \E^Q\left[\int_0^1g(t,q^Q(t))\,dt\right]\quad \text{and} \quad \rho^g(X):=\sup_{Q \in {\cal Q}} \left(\E^Q[X]-\alpha^g(Q)\right).
\end{equation}
Note that $\alpha^g(Q)$ is well defined and takes values in $ \R\cup\{+\infty\}$, as $g$ is bounded from below.

The classical example to keep in mind is the quadratic case, $g(t,q) = \tfrac12|q|^2$. In this case, $\alpha^g$ is nothing but the relative entropy,
\[\textstyle
\alpha^g(Q) = H(Q \, | \, \W) := \E\left[\frac{dQ}{d\W}\log\frac{dQ}{d\W}\right], \ \ Q \in \Q, \label{def:relativeentropy}
\]
and, by Gibbs' variational principle, $\rho^g$ is the cumulant generating functional, $\rho^g(X) = \log \E[e^X]$. A more trivial example is given by $\alpha^ g$ the convex indicator of $0$ and $\rho^g$ the expected value under Wiener measure.
In the following we write $\E$ for expectation under $\W$ and $\E^Q$ for expectation under any other measure $Q$.

One consequence of the assumption (TI) is the stochastic representation of $\rho^g$ in terms of backward stochastic differential equations (BSDE), which we recall: Let $g^*$ stand for the convex conjugate of $g$ in the spatial variable, namely
\begin{equation}
\label{eq gen conjug}	 g^*(t,z):= \sup_{ q \in \R^d}\left(  q \cdot z - g(t, q) \right).
\end{equation}
	Following \cite{DHK1101}, we say that a pair $(Y,Z)$, where $Y$ is a c\`adl\`ag and adapted process and with $Z\in {\cal L}$, is a supersolution to the BSDE (driven by $W$, with terminal condition $X \in L^0(\W)$, and generator $g^*$)
	\begin{equation}\textstyle
	\label{eq:bsde}
		dY(t) = -g^*(t,Z(t))\,dt + Z(t)\,dW(t), \quad Y(1) = X,
	\end{equation}
	if it satisfies
	\begin{equation}\textstyle
		\begin{cases}
 			Y(s)-\int_{s}^{t}g^*(u,Z(u))du+\int_{s}^{t}Z(u) dW(u)\geq Y(t), \quad \text{for every} \quad 0\leq s\leq t\leq 1\\
			\displaystyle Y(1)\geq X,
		\end{cases}
		\label{eq:supersolutions}
	\end{equation}
	and $\int Z\,dW$ is a supermartingale.
A supersolution $(\bar{Y},\bar{Z})$ of \eqref{eq:bsde} is said to be minimal if $\bar{Y}(t)\leq Y(t)$ for all $t \in [0,1]$ and for every other supersolution $(Y,Z)$. By \cite[Theorem 4.17]{DHK1101}, under the condition (TI), the BSDE \eqref{eq:bsde} admits a unique minimal supersolution for every terminal condition $X$ bounded from below.

The crucial link is given in \cite[Theorems 3.4/3.10]{tarpodual}, where it was shown that
\begin{equation}
\label{eq rho = Y}
\rho^g(X) = \bar{Y}(0),
\end{equation}
where $(\bar{Y},\bar{Z})$ is the minimal supersolution of \eqref{eq:supersolutions}, provided that $X$ is {e.g.\ bounded}. This is the aforementioned representation of $\rho^g$ in terms of a BSDE.
Additionally, it is well known that a nonlinear Feynman-Kac formula connects BSDEs with semilinear parabolic PDEs, and we will briefly elaborate on this perspective in Section \ref{sec:lim thm pde} below.
\begin{remark}
\label{rem: solution case}
If $X \in L^\infty(\W)$ and $g$ has at least quadratic growth, {then $g^*$ has subquadratic growth and} the BSDE \eqref{eq:bsde} admits a unique solution $(Y,Z)$ such that $Y$ is bounded (see, e.g., \cite{kobylanski01,Delbaen11}).
Thus, it follows by \cite[Theorem 4.6]{tarpodual} that $\rho^g(X) = Y_0$.
The minimal supersolution and the unique (true) solution coincide.
Consequently, all results stated in this paper for minimal supersolutions transfer to true solutions when $g$ is of superquadratic growth. When this is not the case, a solution to a BSDE need not exist or be unique (see e.g. \citet{Delbaen11}), and the weaker concept of minimal supersolution becomes essential.
\end{remark}

We will derive in Theorem \ref{thm BBD} yet another representation of $\rho^g$, in the spirit of stochastic optimal control. For $F \in L^\infty(\W)$  we show that
\begin{align}\label{eq BBD}\tag{BBD}
\rho^g(F)=\sup_{q \in \L_b} \E\left[ F \left ( W + \int_0^\cdot q(t)dt \right) - \int_0^1 g(t,q(t))dt  \right ].
\end{align}
The fact that $F$ is path-dependent here means that the representation \eqref{eq BBD} does not follow as quickly from the definition of $\rho^g(F)$  as it may seem at first sight.
In the case $g(t,q)=\frac 12 |q|^2$, the representation \eqref{eq BBD} was a key result of \citet{Boue-Dup} and \citet{Lehec}.

In this article we derive limit theorems for the functional $\rho^g$, which, as we have demonstrated, appears naturally in connection to stochastic control, BSDEs, and PDEs.
We first summarize our findings in Section \ref{sec: limit thm} in abstract terms. Then in Section \ref{sec: translation BSDEs} in terms of BSDE. This is followed by Section \ref{sec:schilder.to.schroe} where we present some new insights into the study of convergence of stochastic transport problems (i.e.\ Schr\"odinger-type problems) to optimal transport problems. Crucially, the latter results are obtained as a consequence of the limit theorems. We close this overview section with an outlook discussing connections with PDEs in Section \ref{sec:lim thm pde} and (mean field) optimal control in Section \ref{sec: mf control}.

\subsection{Limit Theorems}
\label{sec: limit thm}

To state our first main limit theorem, a non-exponential version of Sanov theorem in its Laplace principle form, we introduce the following notation:
For a Polish space $E$, we denote by $\P(E)$ the set of Borel probability measures on $E$ equipped with the topology of weak convergence and by $C_b(E)$ the space of bounded continuous functions on $E$.
For $n\in \N$, $k= 1,\dots,n$ and a path $\omega \in \C$, we define the \emph{chopped and rescaled path} $\omega_{(n,k)}\in \C$ by
\begin{equation}
\label{eq:chopped paths}
	\omega_{(n,k)}(t) := \sqrt{n}\left(\omega\Big(\frac{k-1 +t}{n}\Big) - \omega\Big(\frac{k-1}{n}\Big) \right), \quad t\in [0,1].
\end{equation}
Note that $(W_{(n,k)})_{k=1}^n$ are $n$ independent Brownian motions (under $\W$).
In the following, recall that we always work with a given function $g$ satisfying assumption (TI).

\begin{theorem} \label{th:mainlimit}
Define $G_n : [0,1] \times \R^d \rightarrow \R \cup \{\infty\}$ by
\begin{equation*}
	G_n(t,q) := g\left (nt - \lfloor nt\rfloor\,,\, \frac{q}{\sqrt{n}}\right ).
\end{equation*}
{Then $G_n$ satisfies (TI) for each $n$, and} for every $F \in C_b(\P(\C))$ we have
\begin{align*}
\lim_{n\rightarrow\infty}\,\rho^{G_n}\left(F \left(\frac 1n \sum_{k=1}^n\delta_{W_{(n,k)}}\right)\right) 
	&= \sup_{Q \in \Q}\left(F(Q) - \alpha^g(Q)\right) \\
	&= \sup_{q\in {\cal L}_b}\left(F(Q^q) - \E \left[\int_0^1g(t,q(t))dt\right]\right),
\end{align*}
where {$Q^q$ was defined in \eqref{def:Q^q}  for $q \in \L_b$}.
\end{theorem}

The proof is given at the end of Section \ref{sec: proof Limit Thms}.
For the second main result, we adopt the convention that
\begin{align}
	\int_0^1g(t,\dot\omega(t))\,dt = +\infty \label{def:+infty-convention}
\end{align}
whenever $\omega \in \C$ is not absolutely continuous. Define $\C_0:= \{\omega \in \C : \omega(0)=0\}$. {Our second main limit theorem is a non-exponential version of Schilder theorem in Laplace principle form:}

\begin{theorem} \label{co:Schilder-finaltime}
Denote $g_n(t,q):=g(t,q/\sqrt{n} )$. {Then $g_n$ satisfies (TI) for each $n$, and} for every $F \in C_b(\C)$, we have 
\begin{align} \label{eq Schilder familiar}
\lim_{n\to \infty}\,\rho^{g_n}\left(F\left(\frac{W}{\sqrt{n}}\right)\right)= \sup_{\omega \in \C_0} \left (F(\omega) - \int_0^1 g(t,\dot{\omega}(t))dt\right).
\end{align}
Moreover, if $g(t,q)=g(q)$ does not depend on $t$, and if $h \in C_b(\R^d)$, we have
\begin{align*}
\lim _{n\to \infty}\,\rho^{g_n}\left(h\left(\frac{W(1)}{\sqrt{n}}\right)\right)= \sup_{x \in \R^d}(h(x)-g(x)).
\end{align*}
\end{theorem}

The proof is given at the end of Section \ref{sec:var rep}.
Returning to the quadratic case $g(t,q):= \frac 12 |q|^2$ reveals how Theorem \ref{th:mainlimit} and \ref{co:Schilder-finaltime} relate to the classical theorems of Sanov and Schilder.
In this case, $G_n(t,q)=\frac{1}{2n}|q|^2 $ for every $n$, and as mentioned above we get
\[
\rho^{G_n}(X) = \frac{1}{n}\log\E[e^{nX}],
\]
and $\alpha^g(Q)=H(Q|\W)$ is the relative entropy, as defined in \eqref{def:relativeentropy}.
Similarly,
\[
\rho^{g_n}(X) = \frac 1n\log\E[e^{nX}].
\]
Thus in the quadratic case Theorems \ref{th:mainlimit} and \ref{co:Schilder-finaltime} respectively reduce to Sanov and Schilder theorem for Brownian motion in their Laplace principle forms; {see \cite[Theorems 6.2.10 and 5.2.3]{Dem-Zei} respectively for classical statements of Sanov and Schilder's theorems and \cite[Theorems 1.2.1 and 1.2.3]{dupuis-ellis} for the equivalence with Laplace principles}. {For another known example}, if $g$ is the convex indicator of $0$ then $\rho^g$ is just expectation under Wiener measure. In this case Theorem \ref{th:mainlimit} {reduces to the law of large numbers, stating} that the random measures $\frac 1n \sum_{k=1}^n\delta_{W_{(n,k)}}$ converge weakly to $\W$.

{It is important to note that the chopped paths $(W_{(n,k)})_{k=1}^n$ appearing in Theorem \ref{th:mainlimit} cannot be replaced with an arbitrary sequence of $n$ independent Brownian motions, because} the functional $\rho^{G_n}$ is not necessarily law-invariant!\footnote{A functional $\rho:L^0\to \mathbb{R}\cup\{+\infty\}$ is law-invariant if $\rho(X) = \rho(X')$ whenever $X$ and $X'$ have the same law.} For this reason, Theorem \ref{co:Schilder-finaltime} cannot be deduced from Theorem \ref{th:mainlimit}, contrary to the classical case in which Schilder's theorem can be deduced from Sanov's theorem and continuous mapping.
Nevertheless,  in Corollary \ref{co:Schilder-finaltime weird} we derive from Theorem \ref{th:mainlimit} a result more in the spirit of Cram\'er's theorem, which notably shares the same limiting expression as Theorem \ref{co:Schilder-finaltime} despite involving a quite distinct pre-limit quantity.

The key to proving these limit theorems is the stochastic control representation \eqref{eq BBD}, which we establish in Theorem \ref{thm BBD}, as well as the results in \citet{Lac-Sanov}. A major difficulty is the lack of lower-semicontinuity of $\rho^g$ and weak compactness of the sublevel sets of $\alpha^g$ when $g$ is sub-quadratic, which necessitates the study of a better-behaved functional (see $\tilde{\alpha}^g$  in Section \ref{se:compactness-alpha}). 
In Section \ref{se:extensions}, we extend Theorem \ref{co:Schilder-finaltime} to random initial conditions and Theorem \ref{th:mainlimit} to stronger topologies.

\subsection{Scaling limits of BSDE}
\label{sec: translation BSDEs}

In this section, we state two new results on scaling limits for BSDEs. 
Owing to the representation \eqref{eq rho = Y}, these results would follow immediately from Theorems \ref{th:mainlimit} and \ref{co:Schilder-finaltime} if we only considered the value at time zero of such BSDEs. Nonetheless, we are able to bootstrap Theorem \ref{th:mainlimit} and \ref{co:Schilder-finaltime} in order to obtain limits at every time, and not just at time zero. Proofs are deferred to Section \ref{sec BSDE scaled limits}.

\begin{theorem}
\label{thm: Sanov time}
Let $F \in C_b(\P(\C))$, and let  $(Y_n,Z_n)$ be the minimal supersolution of the BSDE
{\begin{equation}
	dY(t) = -{g^*\left( nt - \lfloor nt \rfloor , \sqrt{n}Z(t)\right) } \,dt + Z(t)\,dW(t),\quad  Y(1)=F\left(\frac{1}{n}\sum_{k=1}^n\delta_{W_{(n,k)}}\right ). \label{eq simple BSDE - Sanov}
\end{equation}}
Then, for each $t \in [0,1]$, we have the a.s.\ limit
\begin{align*}
\lim_{n\to\infty}Y_n\left(\frac{\lfloor nt \rfloor}{n}\right) &= {\sup_{Q \in \Q}\left(F(t\W + (1-t)Q) - (1-t)\alpha^g(Q)\right)} \\
	&= \sup_{q \in \L_b}\left(F(t\W + (1-t)Q^q) - (1-t)\,\E\left [\int_0^1g(s,q(s))ds\right ]\right).
\end{align*}
\end{theorem}

We do not know of any results similar to Theorem \ref{thm: Sanov time} in the literature. The limiting expression exhibits a structure remarkably parallel to the Hopf-Lax-Oleinik solution of a Hamilton-Jacobi equation, reviewed in Section \ref{sec PDE Schilder} below. Another interesting feature is that it is a decidedly non-Markovian result; even if $F$ depends only on the time-$1$ marginal of the measure, the terminal conditional {in \eqref{eq simple BSDE - Sanov}} still depends on the value of $W$ at $n$ different times.

Next, in order to state a BSDE analogue of Theorem \ref{co:Schilder-finaltime}, we make use of the following notation: Define $\C_0[t,1]$ to be the set of continuous paths $\omega : [t,1] \to \R$ with $\omega(t)=0$. Also, for $\omega \in \C$ and $\overline\omega \in \C_0[t,1]$, define $\omega \oplus_t \overline\omega \in \C$ by
\[
\omega \oplus_t \overline\omega(s): ={\omega(s\wedge t)} + \overline\omega(s)1_{[t,1]}(s).
\]
Recall as in \eqref{def:+infty-convention} that we set $\int_t^1g(s,\dot{\omega}(s))ds := \infty$ when $\omega$ is not absolutely continuous.

\begin{theorem}\label{thm HL}
Let $F \in C_b(\C)$ and let $(Y_n,Z_n)$ be the minimal supersolution of the BSDE
\begin{equation}\label{eq simple BSDE}
	dY(t) =- g^*\left(t, \sqrt{n} Z(t)\right)\,dt + Z(t)\,dW(t),\quad  Y(1)=  F\left(\frac{W}{\sqrt{n}}\right ).
\end{equation}
Then there exist progressively measurable functions $u_n : [0,1] \times \C \to \R$ such that $Y_n(t) = u_n(t,W/\sqrt{n})$ a.s.\ for each $n$ and $u_n \to u$ pointwise, where
\[
u(t,\omega)  := \sup_{ \overline\omega \in \C_0[t,1] } \left( F( \omega \oplus_t \overline\omega ) - \int_t^1g(s, \dot{\overline\omega}(s))ds \right).
\]
Moreover, for each $t\in[0,1]$, we have the a.s.\ limit
\begin{align*}
\lim_{n\to\infty}Y_n(t) = u(t,0).
\end{align*}
\end{theorem}

The previous {theorem} is noteworthy, as it shows that making the generator of the BSDE explode and its terminal condition trivialize at the same rate gives a non-trivial deterministic limit. Alternatively, we may move the rescaling to the Brownian motion itself. Letting $W^\epsilon=\sqrt{\epsilon}W$ denote Brownian motion with volatility $\epsilon=1/n$, we can rewrite \eqref{eq simple BSDE} as
\begin{align}
dY(t) = -g^*(t, Z(t))dt + Z(t) dW^\epsilon(t), \quad  Y(1) =  F(W^\epsilon), \label{def:BSDE-vanishing-noise}
\end{align} 
and so Theorem \ref{thm HL} also shows a non-trivial effect of ``cooling-down'' the driving Brownian motion in such a BSDE. {The closest related results seem to be those of the form of \cite[Theorem 2.1]{rainero2006principe} on (F)BSDEs with vanishing noise, though the factor $\sqrt{n}$ in $g^*$ in \eqref{eq simple BSDE} is absent in \cite{rainero2006principe}.}

\begin{remark}
The $\epsilon \downarrow 0$ limit of the BSDE \eqref{def:BSDE-vanishing-noise} is intriguing from the perspective of BSDE stability theory. It has been known for some time that if the generator and terminal condition of a BSDE converge in a suitable sense, then so does the solution $(Y,Z)$. Modern BSDE theory has explored similar stability theorems in much more generality, when the driving martingale (in our case, $W^\epsilon$) itself can vary (see \cite{papapantoleon2018stability} and the thesis \cite{saplaouras2017backward} for thorough discussions and references). However, existing results in this direction require either that the limiting martingale (in our case, $0$) has the predictable representation property or that the natural filtrations converge in some sense, both of which clearly fail in our case.
\end{remark}

The factor $\sqrt{n}$ appears in the identity $Y_n(t)=u_n(t,W/\sqrt{n})$ in Theorem \ref{thm HL} for two reasons. On a purely mathematical level, this provides the scaling that results in a random ($\omega$-dependent) limit for $u_n$. The second and more practical reason is that one can interpret $u(t,\omega)$ as the value function of a stochastic control problem in which the state process, $W/\sqrt{n}$, is observed up to time $t$ to agree with the path $(\omega(s))_{s \le t}$. This will be perhaps more clear when we reinterpret Theorem \ref{thm HL} in terms of PDEs in Section \ref{sec PDE Schilder} below.

In the quadratic case $g(t,q)=\frac 12 |q|^2$, Theorem \ref{thm HL} reads as a ``conditional'' version of Schilder's theorem for Brownian motion. Indeed, the solution of BSDE \eqref{eq simple BSDE} and its a.s.\ limit in Theorem \ref{thm HL} are given by
\begin{align*}
Y_n(t) &= \frac{1}{n}\log\E[\exp(nF(W/\sqrt{n})) \, | \, \F_t] \\
	&\rightarrow  \sup_{\omega \in \C_0[t,1]}\left( F( 0 \oplus_t \omega ) - \frac 12\int_t^1 |\dot{\omega}(s)|^2ds \right).
\end{align*}
Of course, it is straightforward to derive this directly from the usual form of Schilder's theorem. Similarly, in the quadratic case, Theorem \ref{thm: Sanov time} can be rewritten as the a.s.\ limit
\begin{align*}
\lim_{n\rightarrow\infty}\frac{1}{n}\log\E\left[\left.\exp\left(nF\left(\frac{1}{n}\sum_{k=1}^n\delta_{W_{(n,k)}}\right )\right) \, \right| \F_{\frac{\lfloor nt \rfloor}{n}}\right] = \sup_{Q \in \Q}(F(t\W + (1-t)Q) + (1-t)H(Q | \W)),
\end{align*}
for each $t \in [0,1]$. It is likely that a direct argument would yield in this case that the same holds even if we replace $\frac{\lfloor nt \rfloor}{n}$ with $t$ on the left-hand side. More generally, we conjecture that $Y_n(t)$ converges to the same limit as $Y_n\left(\frac{\lfloor nt \rfloor}{n}\right)$ in the setting of Theorem \ref{thm: Sanov time}.

\subsection{Small noise limit of Schr\"odinger-type problems}
\label{sec:schilder.to.schroe}
{
It was established by \citet{Mi04} that classical quadratic optimal transport is the small-noise limit of so-called Schr\"odinger problems. This was then extended by \citet{MiTh08} to non-quadratic situations. In this case, optimal transport is obtained as a small-noise limit of a \emph{stochastic transport problem}. This latter stochastic variant can be interpreted as a non-exponential, Schr\"odinger-type problem. The method employed by the authors relies on PDE techniques and the given Brownian setting that they propose. On the other hand, \citet{Le12} extended these considerations to a non-Brownian setting by employing large deviations arguments instead of PDEs; his is therefore a fully probabilistic approach, which was further developed in \cite{Le14,Le16}. However, the approach of L\'eonard, when applied in the aforementioned Brownian setting of Mikami-Theullien, can only cover the quadratic case.

In this part we aim to deepen the study of optimal transport as a small-noise limit of stochastic optimal transport. Let us present our setting and main result. After this has been done, we will comment more on how it connects to priori literature.

For $\epsilon>0$ we introduce the set {$\P^*_\epsilon(\C)$ of $Q \in \P(\C)$ for which there exists a progressively measurable $\R^d$-valued processes $q^Q$ such that the process 
\[
\frac 1{ \sqrt{\epsilon}}\left(W(t) - W(0)-\int_0^t q^Q(s)ds\right)
\]
is a standard $d$-dimensional Brownian motion under $Q$.} We stress that for $Q\in \mathcal{P}^*_\epsilon(\C)$ the process $q^Q$ is uniquely determined (in the $dt\otimes dQ$-a.s.\ sense), and that it is understood in the above definition that $q^Q$ is $dt$-integrable $Q$-a.s.
We denote by $Q^t$ the marginal at time $t$ of a path measure $Q\in\mathcal{P}(\C)$, and by $\pi^i$ the $i$-th marginal of a measure $\pi\in \mathcal{P}(\R^d\times \R^d)$. 

We now introduce the problems of interest in this part of the work. Let $Z$ be a separable Banach space, which we endow with its Borel sigma-algebra. We are given an \emph{observable} $H$, which is nothing more than a continuous linear operator
$$H:\C\to Z\, .$$
For $\mu$ and $\nu$ Borel probability measures respectively on $\R^d$ and $Z$, we examine here the problems
\begin{align}
\inf_{\substack{Q\in\mathcal{P}^*_\epsilon(\C) \\ Q^0=\mu,\,Q\circ H^{-1}=\nu}} E^Q\left [ \int_0^1 g\left (t,q^Q(t)\right )dt \right ], \label{eq path space with H}
\end{align}
and their limits when $\epsilon \downarrow 0$.

We stress that taking the \emph{classical observable}, namely $$Z=\R^d\text{ and }H(\omega)=\omega(1),$$  corresponds to the situation of Schr\"odinger problems, modulo the fact that $g$ need not be the quadratic function, see e.g.\ the survey \citet{Le14}. In general we think of $H$ as an \emph{observable random quantity} whose distribution $\nu$ we know, and impose in advance into the problem. For instance, $H$ could give the value of a path at different time points, as well as the value of successive integrals of the path.

We now state our first main result, which relies fundamentally on our Schilder-type result Theorem \ref{co:Schilder-finaltime}. {In the following, let $\W_\epsilon$ denote Wiener measure with volatility $\epsilon$, i.e., $\W_\epsilon = \W \circ (\sqrt{\epsilon}W)^{-1}$.}
Notice that in the quadratic case, the optimization problem \eqref{eq path space with H} reduces to an optimization over absolutely continuous probability measures.
Proofs of the results announced in this section are given in Section \ref{sec Schroedinger details}.

\begin{corollary}\label{coro Schroedinger}
Let $\mu\in \mathcal{P}(\R^d)$ and $\nu\in \mathcal{P}(Z)$, and let $H : \C \rightarrow Z$ be linear and continuous. Let $\nu_\epsilon : = \nu * \left (\, P_\epsilon\circ H^{-1} \right )$ denote
the convolution of $\nu$ with the push-forward by $H$ of $\W_\epsilon$. Then\footnote{We take the convention that infimum over an empty set equals $+\infty$.} 
\begin{align}
\lim_{\epsilon\downarrow 0}\, \inf_{\substack{Q\in\mathcal{P}^*_\epsilon(\C) \\ Q^0=\mu,\,Q\circ H^{-1}=\nu_\epsilon}} \E^Q\left [ \int_0^1 g\left (t,q^Q(t)\right )dt \right ] & = \inf_{\substack{Q\in\mathcal{P}(\C)\\ Q^0=\mu,\,Q\circ H^{-1}=\nu}} \E^Q\left [ \int_0^1 g\left (t,\dot{W}(t)\right )dt \right ]\label{eqMikamiLeo1}.
\end{align}
 Furthermore, we have:
 \begin{itemize}
 \item The problem 
\begin{align}\inf_{\substack{Q\in\mathcal{P}^*_\epsilon(\C) \\ Q^0=\mu,\,Q\circ H^{-1}=\nu_\epsilon}} \E^Q\left [ \int_0^1 g\left (t,q^Q(t)\right )dt \right ],\label{eq prob lhs eps}
\end{align} has an optimizer as soon as $\{Q\in \mathcal{P}^*_\epsilon(\C):\, Q^0=\mu,\,Q\circ H^{-1}=\nu_\epsilon\}\neq \emptyset$. Analogously,  
\begin{align}\inf_{\substack{Q\in\mathcal{P}(\C)\\ Q^0=\mu,\,Q\circ H^{-1}=\nu}} \E^Q\left [ \int_0^1 g\left (t,\dot{W}(t)\right )dt \right ],\label{eq prob rhs 0}
\end{align}
 has an optimizer as soon as  $\{Q\in \mathcal{P}(\C):\, Q^0=\mu,\,Q\circ H^{-1}=\nu\}\neq \emptyset$.
 \item If for all $\epsilon>0$ small, an optimizer $Q_\epsilon$ of \eqref{eq prob lhs eps} exists, then any cluster point of $\{Q_\epsilon\}_\epsilon$ is an optimizer of \eqref{eq prob rhs 0}. In particular, if the latter problem has a unique optimizer, then any cluster point of $\{Q_\epsilon\}_\epsilon$ is equal to it.
 \end{itemize} 
\end{corollary}

This result is close to \cite[Theorem 3.2]{MiTh08}, except for the important facts: We can handle here a time-dependent function $g$ with nearly no regularity assumptions, and we take a rather general observable $H$ rather than just the classical one $H(\omega)=\omega(1)$. 

We stress that introducing the \emph{mollified measures} $\nu_\epsilon$ is in general unavoidable. For instance, in the quadratic case, with $H(\omega)=\omega(1)$ and when $\mu$ and $\nu$ are discrete, the value in the l.h.s.\ of \eqref{eqMikamiLeo1} is $+\infty$ whereas the right-hand side could very well be finite. In our opinion this mollification unnecessarily blurs the elegant story ``stochastic transport converges to optimal transport.'' As \cite[Proposition 2.1]{Mi04} explains in the quadratic case with the classical observable, the PDE-based approach may still work without such mollification at the expense of more restrictive assumptions. Our second main result on the matter is to show that when $g$ is strictly less than quadratic, then the mollification can be fully avoided without any further assumptions.

\begin{corollary}\label{coro g sub quadratic}
Let us assume that $g(t,q)=g(q)$ and consider the classical case $$Z=\R^d \text{ and }H(\omega)=\omega(1).$$ Assume the existence of  the following limit for some $r\in(0,2)$:
$$\lim_{|q|\to\infty}\frac{g(q)}{|q|^r}\in (0,+\infty).$$
 Then all the conclusions of Corollary \ref{coro Schroedinger} are valid when we take $\nu_\epsilon\equiv\nu$ for all $\epsilon$.
\end{corollary}

The curious reader may wonder why the need of $g$ being strictly less than quadratic. The reason is that we make crucial use of Brownian bridges in the proof, and that these measures have non-square-integrable drifts.\\

Let us finally summarize, and classify, our contribution into the subject:
\begin{itemize}
\item \emph{Methodological}: We take inspiration in L\'eonard's fully probabilistic approach of \cite{Le12}, and crucially use our generalized Schilder-type result (Theorem \ref{co:Schilder-finaltime}) in order to study the problem.
\item \emph{Technical}: We can handle a time-dependent function $g$ under almost no regularity assumptions. This is a first advantage of not working with PDEs.
\item \emph{Novelty}: We are able to consider in the definition of the problem quite general \emph{observables}. This is, in our opinion, out of the scope of PDE methods. Our second finding, explained in Corollary \ref{coro g sub quadratic}, provides a much more transparent form to the statement that ``optimal transport is a small-noise limit of stochastic optimal transport,'' since we can avoid the {pre-limit mollification} at no regularity cost.
\end{itemize}

{
The reader who is mostly interested in these results on the Schr\"odinger-type problem, may directly go to Section \ref{sec Schroedinger details}. All technical prerequisites are covered in Section \ref{sec extension of results}.}

\subsection{Connections with PDEs}
\label{sec:lim thm pde}
In this subsection we specialize the limit theorems to functions $F$ on $\P(\C)$ (resp. $\C$) which depend only on the time-$1$ marginal of the measure (resp. the time-$1$ value of the path).
In this case, the so-called \emph{nonlinear Feynman-Kac formula} (see the recent book \cite[Section 5.1.3]{zhangbook} for a typical case) allows to reinterpret the BSDE results of Section \ref{sec: translation BSDEs} in terms of semilinear parabolic partial differential equations.

\subsubsection{A PDE form of Theorem \ref{thm HL}}
\label{sec PDE Schilder}
As a first special case, suppose the function $F$ in Theorem \ref{thm HL} depends only on the final value of the path; that is, $F(w)=f(w(1))$ for all $w \in \C$, for some $f \in C_b(\R^d)$. Then, according to \cite[Theorem 5.2]{Drap-Main16}, we can write $Y_n(t) = v_n(t,W(t))$, where $v_n : [0,1] \times \R^d \rightarrow \R$ solves (i.e.\ is the minimal viscosity supersolution of) the PDE
\begin{equation}
		\begin{cases}
			\partial_tv_n(t,x) + \frac 12 \Delta v_n(t,x) + g^*(t, \frac{1}{\sqrt{n}}\nabla v_n(t,x)) = 0 \quad \text{on } [0,1]\times \R^d\\
			 v_n(1, x) = f(\frac{x}{\sqrt{n}}), \quad \text{for } x\in \R^d,
		\end{cases}
\end{equation}
where the gradient and Laplacian operators act on the $x$ variable. Alternatively, defining $u_n(t,x) = v_n(t,\sqrt{n}x)$, we find that $u_n$ should solve the PDE
\begin{equation*}
		\begin{cases}
			\partial_tu_n(t,x) + \frac{1}{2n} \Delta u_n(t,x) + g^*(t, \nabla u_n(t,x)) = 0 \quad \text{on } [0,1]\times \R^d\\
			 u_n(1, x) = f(x), \quad \text{for } x\in \R^d.
		\end{cases}
\end{equation*}
In this PDE, the factor $n$ appears only in the denominator of the diffusion coefficient, and as $n\rightarrow\infty$ we expect $u_n$ to converge to the solution $u$ of the first-order  PDE
\begin{equation}\label{eq HJ u}
		\begin{cases}
			\partial_tu(t,x) + g^*(t, \nabla u(t,x)) = 0 \quad \text{on } [0,1]\times \R^d\\
			 u(1, x) = f(x), \quad \text{for } x\in \R^d.
		\end{cases}
\end{equation}
If $g(t,x)=g(x)$ is time-independent, the solution should be given by the Hopf-Lax-Oleinik formula,
\begin{equation}
\label{eq:hopf lax form}
u(t,x) = \sup_{y \in \R^d}\left(f(y) - (1-t)g\left(\frac{y-x}{1-t}\right)\right).
\end{equation}
We then obtain
\[
\lim_{n\rightarrow\infty}v_n(0,0) = \lim_{n\rightarrow\infty}u_n(0,0) = u(0,0) = \sup_{x \in \R^d}(f(x)-g(x)),
\]
which agrees with the limiting expressions Theorems \ref{co:Schilder-finaltime} and \ref{thm HL}. We will expand and formalize these heuristics in Proposition \ref{prop:hopf-lax} below. {Noting that $Y_n(t,\omega) = u_n(t,\omega(t)/\sqrt{n})$, this explains the choice of scaling in the first claimed limit of Theorem \ref{thm HL}.}

\subsubsection{Path-dependent PDEs}  
It is tempting to search for a PDE formulation of Theorem \ref{thm HL}, analogous to the discussion in Section \ref{sec PDE Schilder}. Indeed, the quantity $u_n(t,\omega)$ in Theorem \ref{thm HL} can be viewed as the value function of a stochastic control problem with a path-dependent objective functional, and Theorem \ref{thm HL} identifies the limiting function $u(t,\omega)$ as itself the value of a deterministic control problem. In analogy with Section \ref{sec PDE Schilder}, we speculate that Theorem \ref{thm HL} could be rewritten as a vanishing viscosity limit of path-dependent Hamilton-Jacobi-Bellman equations, but this is beyond the scope of this paper. Refer to \cite{lukoyanov2007viscosity,EKTZ2014,bayraktar2018path}
 and the references therein for relevant literature on path-dependent PDEs and particularly to \cite{ma2016large} where a connection with large deviations appears.

\subsubsection{A PDE form of Theorem \ref{thm: Sanov time}}\label{sex PDE Sanov}
In the general context of Theorem \ref{thm: Sanov time}, when $F \in C_b(\P(\C))$ depends on the whole path, the BSDE of Theorem \ref{thm: Sanov time} cannot be expressed using PDEs. However, when $F$ depends only on the marginal law at the final time, i.e., $F\ = F(m(1))$ for some $F \in C_b(\P(\R^d))$, a different PDE representation is available. The terminal condition in the BSDE of Theorem \ref{thm: Sanov time} becomes
\[
F\left(\frac{1}{n}\sum_{k=1}^n\delta_{W_{(n,k)}(1)}\right) = F\left(\frac{1}{n}\sum_{k=1}^n\delta_{ \sqrt{n}(W(k/n) - W((k-1)/n)) }\right).
\]
This  terminal condition depends on the path of $W$ only through the values of $W(t)$ at the finitely many time points $t=1/n,2/n,\ldots,1$. Hence, the BSDE of Theorem \ref{thm: Sanov time} can be seen as a concatenation of $n$ Markovian BSDEs, each of which can be represented by a PDE.

More details will be given in Section \ref{sec BSDE}, specifically in Proposition \ref{pr:PDEiteration}, but let us briefly summarize the idea.
Define an operator $\mathbb L_n$, taking lower semicontinuous lower bounded functions of $(\R^d)^n$ to lower semicontinuous lower bounded functions of $(\R^d)^{n-1}$, as follows: Given $f:(\R^d)^n \to \R$ and $(x_1,\ldots,x_{n-1}) \in (\R^d)^{n-1}$, we define $\mathbb L_nf(x_1,\dots,x_{n-1}) := v(0,0)$, where $v=v(t,x)$ is the the minimal viscosity supersolution of the PDE
\begin{equation*}
		\begin{cases}
			\partial_t v(t,x) + \frac 12 \Delta v(t,x) + g^*(t,\nabla v(t,x)) = 0 \quad \text{on } [0,1]\times \R^{d}\\
			 v(1, x) = f(x_1,\ldots,x_{n-1},x), \quad \text{for } x\in \R^d.
		\end{cases}
\end{equation*}
For $n=1$, we interpret $\mathbb L_1$ as mapping from functions of $\R^d$ to real numbers. The composition $\mathbb L_1\cdots\mathbb L_{n-1}\mathbb L_n$ then maps a function of $(\R^d)^n$ to a real number. For $F \in C_b(\P(\R^d))$, define $F^n : (\R^d)^n \rightarrow \R$ by
\[
F^n(x_1,\dots,x_n):= F\left( \frac{1}{n}\sum_{i=1}^n\delta_{x_i} \right)
\]
Then we have $\frac{1}{n}\mathbb L_1\cdots\mathbb L_{n-1}\mathbb L_n(nF^n) = Y_n(0)$, where $Y_n$ is as in Theorem \ref{thm: Sanov time}.

\subsection{Approximation schemes for (mean field) optimal control problems}\label{sec: mf control}

Interpreting the quantity $\rho^g(F)$ as well as the limiting expressions of Section 
\ref{sec: limit thm} as the values of optimal control problems suggests certain numerical schemes, for both mean field stochastic control problems and non-convex deterministic optimal control problems. {We stress that by allowing the function $g$ to be $+\infty$-valued, we can induce pointwise control constraints in these problems.}

\subsubsection{Mean field stochastic optimal control}

The limiting quantity in Theorem \ref{th:mainlimit}, or in  Theorem \ref{thm: Sanov time}, is a stochastic optimal control problem of mean field type. Indeed, one may express this limit quantity as
\begin{align}
\sup_{q \in \L_b} \left\{ F\left(\W \circ (X^q)^{-1}\right) - \E\left[\int_0^1g(t,q(t))dt\right] \right\}. \label{def:MFcontrol}
\end{align}
where we define 
\begin{align}
X^q(t) = \int_0^tq(s)ds + W(t). \label{def:MFcontrol-X}
\end{align}
This kind of optimization problem has been the subject of active research in recent years, with most of the literature focused on solution techniques, using either maximum principles \cite{AnderssonDjehiche,CarmonaDelarue-MVcontrol} or infinite-dimensional Hamilton-Jacobi-Bellman equations \cite{PhamWei-MVcontrol,LaurierePironneau}. Often, in this literature, the function $g$ or the coefficients of the SDE for $X$ may depend additionally on $X$ and even its law. In this sense, we encounter in this paper only  a rather special type of mean field control problem, but one which nonetheless includes many noteworthy examples, such as mean-variance optimization problems.

A mean field control problem such as \eqref{def:MFcontrol} arises heuristically as an $n\rightarrow\infty$ (mean field) limit of an optimal control problem consisting of $n$ state processes, described loosely as follows:
\begin{align}
\sup_{(q_1,\ldots,q_n)}\E\left[F\left(\frac{1}{n}\sum_{k=1}^n\delta_{X_k}\right) - \frac{1}{n}\sum_{k=1}^n\int_0^1g(t,q_k(t))dt\right], \label{def:MFcontrol-particleapprox}
\end{align}
where the supremum is over progressively measurable square-integrable processes $q_k$, with the state processes $X_k$ defined by
\[
X_k(t) = \int_0^tq_k(s)ds + W_k(t),
\]
for independent Brownian motions $W_1,\ldots,W_n$ defined on some probability space. The optimal value in \eqref{def:MFcontrol-particleapprox} should converge to the optimal value in \eqref{def:MFcontrol}, as was rigorously justified only recently in \cite{Lacker-MVcontrol}, at least for certain functions $F$.
The $n$-particle control problem \eqref{def:MFcontrol-particleapprox} is arguably more amenable to numerical approximation than the mean field counterpart \eqref{def:MFcontrol}, as (finite-dimensional) dynamic programming and PDE methods are available for the former.

Interestingly, our Theorem \ref{th:mainlimit} provides an alternative approximation for \eqref{def:MFcontrol} which could presumably be the basis for a numerical scheme. In particular, the pre-limit expression in Theorem \ref{th:mainlimit} can be written as the value of a stochastic control problem:
\begin{align}
\sup_{q \in \L_b} \E\left[ F\left(\frac{1}{n}\sum_{k=1}^n\delta_{X^q_{(n,k)}}\right) - \int_0^1g\left(nt - \lfloor nt \rfloor, \, \frac{q(t)}{\sqrt{n}}\right)dt\right], \label{def:MFcontrol-ourapproximation}
\end{align}
with $X^q$ as in \eqref{def:MFcontrol-X}. The key advantage, compared to the $n$-particle approximation of the previous paragraph, is that here there is only one controlled process. The tradeoff, however, is that the control problem \eqref{def:MFcontrol-ourapproximation} is inevitably highly path-dependent. If we assume $F \in C_b(\P(\C))$ depends only on the time-$1$ marginal of the measure, then the $n$-particle problem \eqref{def:MFcontrol-particleapprox} becomes Markovian, whereas our approximation \eqref{def:MFcontrol-ourapproximation} remains path-dependent, as the cost function depends on the value of the state process at the $n$ grid points $(X^q(1/n),X^q(2/n),\ldots,X^q(1))$.
We discussed in Section \ref{sex PDE Sanov} (with full details to come in Section \ref{sec BSDE}) how one can essentially still apply dynamic programming and PDE methods to this kind of non-Markovian control problem.

Intuitively, the two approximations \eqref{def:MFcontrol-particleapprox} and \eqref{def:MFcontrol-ourapproximation} may appear more closely related than they truly are. On the one hand, in \eqref{def:MFcontrol-ourapproximation}, we may interpret $X_{(n,k)}$ for $k=1,\ldots,n$ as playing the role of the $n$ particles in \eqref{def:MFcontrol-particleapprox}. Indeed, these chopped paths are driven by independent Brownian motions. However, in \eqref{def:MFcontrol-ourapproximation}, the control $q(t)$ in the time interval $t \in [k/n,(k+1)/n]$ is allowed to depend on the entire past of the process $(X_s)_{s \le t}$, which includes the entire paths $(X_{(n,1)},\ldots,X_{(n,k)})$ on the entire interval $[0,1]$. On the other hand, in \eqref{def:MFcontrol-particleapprox}, the control $q_k(t)$ of particle $k$ at time $t$ can depend only on the paths of the particles up to time $t$, or $(X_1(s),\ldots,X_n(s))_{s \le t}$.

\subsubsection{A probabilistic numerical scheme for non-convex optimization}
The limit theorems presented above suggest a probabilistic method for the numerical computation of the value of optimization problems of the form
\begin{equation*}
	V:=\sup_{\omega \in \C}\left(F(\omega) - \int_0^1g(t,\dot \omega_t)\,dt \right)
\end{equation*}
where $F\in C_b(\C)$ and $g:[0,1]\times \R^d\to \R_+$ is convex for each $t \in [0,1]$.
In fact, approximating the value of this (deterministic) non-convex optimization problem can be very challenging due to the fact that there might be several local suprema.
Theorem \ref{thm HL} shows that if $\rho^{g_n}(F(W/\sqrt{n}))$, the minimal supersolution of the BSDE with terminal condition $F(W/\sqrt{n})$ and generator $g_n^*$ can be computed, then $V$ is obtained by taking $n$ large enough.
This can serve as a test to verify whether a local supremum (as obtained via a numerical scheme) is an actual global supremum. Numerical approximations of BSDEs, for sufficiently nice generators, are well understood.
We refer for instance to \cite{Bou-Tou04,Gob-Lem-War05,Gob-Tur16} and the references therein.

\subsection{Outline of the remainder of the paper}
The rest of the paper is devoted to proving the results stated above. First, Section \ref{sec:var rep} proves the variational formula \eqref{eq BBD} and then uses it to prove Theorem \ref{co:Schilder-finaltime}. Section \ref{sec: proof Limit Thms} gives the more involved proof of Theorem \ref{th:mainlimit}. The remaining four sections address the applications, beginning with BSDEs and PDEs in Sections \ref{sec BSDE scaled limits} and \ref{sec BSDE}, respectively. Section \ref{se:extensions} gives some modest extensions of our main results, in particular to allow for non-random initial states, which is crucial in proving our results on Schr\"odinger problems in the final Section \ref{sec Schroedinger details}.

\section{The stochastic control representation}
\label{sec:var rep}

This section is devoted to the stochastic control representation of $\rho^g$, already hinted at in \eqref{eq BBD}. In fact, we will establish a stronger result. In the following, the total variation metric on $\P(\C)$ is defined by $(Q,Q') \mapsto \sup \int f\,d(Q-Q')$, where the supremum is over measurable functions $f : \C \to [-1,1]$.

\begin{theorem}\label{thm BBD}
Let $H:\mathcal P(\C)\to\R$ be bounded and continuous with respect to total variation, then
\begin{align}\label{rem BBD measures}
\sup_{Q\in\mathcal Q}\left\{H(Q)- \alpha^g(Q)\right\}=\sup_{q\in\mathcal L_b}\left\{H(Q^q)- \E\left [\int_0^1 g(t,q(t))dt\right ]\right\},
\end{align}
where $$\textstyle Q^q := \W \circ \left( W + \int_0^\cdot q(t)dt\right)^{-1}.$$ In particular, if $F: \C\to\R$ is Borel measurable and bounded, then 
\begin{align}\tag{BBD}
\rho^g(F)=\sup_{q \in \L_b} \E\left[ F \left ( W + \int_0^\cdot q(t) \,dt  \right) - \int_0^1 g(t,q(t))dt  \right ].
\end{align}
\end{theorem}

Recall that in the quadratic case $g(t,q):=|q|^2/2$ we have $\rho^g(X)=\log\E[e^X]$, and Equation \eqref{eq BBD} becomes the celebrated  variational principle obtained in \cite{Fl78,Boue-Dup,Bo00}. We stress that in such case, \eqref{eq BBD} has already proved to be a powerful tool in stochastic analysis, e.g.\ in large deviations theory \cite{Boue-Dup}, in convex geometry (e.g.\ functional inequalities \cite{Lehec}) and in the study of convexity properties of Gaussian measure \cite{Bo00,vanHandel17}. For these reasons we employ the name \emph{Borell-Bou\'e-Dupuis formula} for the representation \eqref{eq BBD}. On the other hand, for nonlinear $H$, the identity \eqref{rem BBD measures} seems to be novel even in the quadratic case and will be useful in the proofs of Theorem \ref{th:mainlimit} and \ref{co:Schilder-finaltime}.

For the stochastic control connoisseur we stress that the formula \eqref{eq BBD} is a natural consequence of the definition of $\rho^g$ (see \ref{eq def alpha rho}) and the fact that optimizing over open-loop or closed-loop controls should yield the same optimal value. The difficulty lies mainly in the rather arbitrary path-dependence of $F$. 

We prepare with a lemma which allows us to restrict the supremum in the definition of $\rho^g$ to a more convenient class. In the following, recall that $\L_b$ denotes the set of bounded progressively measurable functions $q : [0,1] \times \C \rightarrow \R^d$. Let $\L_b^s$ denote the set of $q \in \L_b$ such that the SDE
\begin{equation}
\label{eq:SDE_simpledrift}
	dX(t) = q(t,X)dt + dW(t), \quad X(0)=0,
\end{equation}
 admits a unique strong solution. If $Q$ denotes the law of $X$ then $q=q^Q$. We find it useful, and intuitive, to overload the notation $\rho^g$ in the following way: if $H:\mathcal P(\C)\to\R$ we write
$$\rho^g(H):=\sup_{Q\in\mathcal Q}\left\{H(Q)- \alpha^g(Q)\right\}.$$
This notation is only employed within this section of the article.

\begin{lemma}\label{lem useful SDE}
Let $H:\mathcal P(\C)\to\R$ be as stated in Theorem \ref{thm BBD}. We have
\begin{align}\label{eq:sub q SDE}
\rho^g(H)=\sup\left\{H(Q) -  \E^Q\left[ \int_0^1 g(t,q^Q(t))dt  \right ] : Q \in \Q, \, q^Q \in \L_b^s\right\},
\end{align}
\end{lemma}
\begin{proof}
As $g$ is bounded from below, we may assume without loss of generality that $g \ge 0$, by making an additive shift to both $H$ and $g$.
We make two intermediate approximations. First, define $\Q_\infty$ to be the set of $Q \in \Q$ such that $\int_0^1 g(t,q^Q(t))dt \in L^\infty(\W)$.
Let us show
\begin{align}\label{eq:sub q simple1}\textstyle
\rho^g(H)=\sup_{Q \in \Q_\infty} \left\{ H(Q)-\E^Q\left[  \int_0^1 g(t,q^Q(t))dt  \right ]\right\}.
\end{align}
To prove this, we first note that we may trivially restrict the supremum in the definition of $\rho^g(F)$ to those $Q \in \Q$ for which $\E^Q\int_0^1g(t,q^Q(t))dt < \infty$. Fix one such $Q \in \Q$. In the notation of (TI), we have $q^Q(t) \in \mathrm{dom}(g(t,\cdot))$, $dt \otimes d\W$-a.e. Let $\tau_n = \inf\{t : \int_0^tg(s,q^Q(s))ds > n\} \wedge 1$ and define $\frac{dQ_n}{d\W}:=\E[\frac{dQ}{d\W}\,|\,\F_{\tau_n}]$, so that $q^{Q_n}=q_n$, where $q_n(t) := q^Q(t)1_{\{t \le \tau_n\}}$.
We easily check that $dQ_n/d\W \rightarrow dQ/d\W$ in probability, and, by Scheffe's lemma, in $L^1(\W)$. This implies that $Q_n\to Q$ in total variation, and so $H(Q^n) \rightarrow H(Q)$. Moreover, $Q_n=Q$ on $\F_{\tau_n}$, and we deduce
\begin{align*}\textstyle
H(Q^n) -\E^{Q_n}\left[ \int_0^1g(t, q^{Q_n}(t))\,dt\right] & \textstyle= H(Q^n) - \E^Q\left[\int_0^{\tau_n}g(t, q(t))\,dt\right] - \E^{Q_n}\left[\int_{\tau_n}^1g(t,0)\,dt\right] \\ \textstyle
	& \textstyle\rightarrow H(Q)- \E^Q\left[ \int_0^1g(t, q(t))\,dt\right].
\end{align*}

With \eqref{eq:sub q simple1} established, we next show that in fact
\begin{align}\label{eq:sub q simple2}\textstyle
\rho^g(H)=\sup\left\{ H(Q)-\E^Q\left[  \int_0^1 g(t,q^Q(t))dt  \right ]  : Q \in \Q_\infty, \, q^Q \in \L_b\right\}.
\end{align}
To prove this, fix $Q \in \Q_\infty$. We again have $q^Q(t) \in \mathrm{dom}(g(t,\cdot))$, $dt \otimes d\W$-a.e. Define $q_n(t)$ as the projection of $q^Q(t)$ onto the centered ball of radius $n$, that is:
\[
q_n(t) := q^Q(t)1_{\{|q^Q(t)| \le n\}} + \frac{n}{|q^Q(t)|}q^Q(t)1_{\{|q^Q(t)| > n\}}.
\]
Using convexity of $g(t,\cdot)$ and $g \ge 0$, we have
\begin{align}
g(t,q_n(t)) &\le g(t,0) + g(t,q^Q(t)). \label{pf: BBD SDE lem 1}
\end{align}
For each $(t,\omega)$ it holds for all sufficiently large $n$ that $q_n(t,\omega)=q^Q(t,\omega)$, and thus $g(t,q_n(t,\omega)) \rightarrow g(t,q^Q(t,\omega))$ pointwise. Find $Q_n \in \Q$ such that $q^{Q_n}=q_n$. Since $q_n \rightarrow q^Q$, we deduce, as in the previous step, that $dQ_n/d\W \rightarrow dQ/d\W$ in $L^1(\W)$ and thus $Q_n \to Q$ in total variation. Thanks to \eqref{pf: BBD SDE lem 1} we may apply dominated convergence to get
\begin{align*}\textstyle
H(Q_n)-\E^{Q_n}\left[ \int_0^1g(t, q^{Q_n}(t))\,dt\right] \rightarrow H(Q)-\E^{Q}\left[ \int_0^1g(t, q^Q(t))\,dt\right].
\end{align*}

Now that we have proven \eqref{eq:sub q simple2}, we show as a final approximation that
\begin{align}\label{eq:sub q simple}\textstyle
\rho^g(H)=\sup \left\{ H(Q)-\E^Q\left[  \int_0^1 g(t,q^Q(t))dt  \right ]\right\},
\end{align}
where the supremum is taken over $Q \in \Q$ such that $q^Q$ is a simple process.
We say here that $q : [0,1] \times \C \rightarrow \R^d$ is a \emph{simple process} if there is a (deterministic) partition $t_0<t_1<\dots<t_N$ and bounded ${\cal F}_{t_i}$-measurable random variables $\xi_i$ for which
\begin{equation*}\textstyle
	q(t) = \xi_01_{\{0\}}(t) + \sum_{i=0}^{N-1}\xi_i1_{(t_i,t_{i+1}]}(t).
\end{equation*}
We start from \eqref{eq:sub q simple2}. Fix $Q \in \Q_\infty$ such that $q^Q \in \L_b$, noting that necessarily $q^Q(t) \in \mathrm{dom}(g(t,\cdot) )$ $dt\otimes d\W$-a.e. Suppose $|q^Q| \le C$ pointwise, where $C <\infty$.
Due to convexity and lower semicontinuity of $g(t,\cdot)$, upon making the further approximation $q_\epsilon(t):=\epsilon\bar q(t) + (1-\epsilon)q^Q(t)$, with $\epsilon\in (0,1)$ and for  $\bar q\equiv 0\in \text{ri}(\text{dom}(g(t,\cdot)))=:\Rel$, we can assume $q^Q(t) \in \Rel$. The convex set  $\Rel$ is, by assumption, independent of the time $t$. We now show that $q^Q$ can be suitably approximated by measurable processes with continuous paths. First remark that $q^Q$ can be identified with a measurable function on $$E:=\Psi([0,1]\times \C),$$
where $\Psi(t,\omega):=(t,\omega(\cdot\wedge t))$. The space $E$ is Polish, as a closed subset of the Polish space $[0,1]\times \C$. By Lusin's Theorem, there is for every $k$ a closed set $E^k\subset E$ such that $q^Q$ restricted to $E^k$ is continuous and $dt\otimes d\W(E^k)\leq 2^{-k}$. By the Tietze extension theorem \cite[Theorem 4.1]{dugundji1951extension}, we can find a continuous function $q_k$ on $E$ which coincides with $q^Q$ when restricted to $E^k$ and which takes values in the closed convex hull of $\{q^Q(t,\omega) : (t,\omega) \in E\}$. In particular, $q_k(t,\omega) \in \Rel$ and $|q_k(t,\omega)| \le C$ for each $(t,\omega)$. By Borel-Cantelli, $q_k$ converges $dt\otimes d\W$-a.s.\ to $q^Q$. By further approximating each $q_k$, we may obtain the existence of a sequence of simple processes converging $dt\otimes d\W$-a.s.\ to $q^Q$, each of which still takes values in $\Rel$ and is bounded uniformly by $C$. Let us re-brand by $q_n$ this sequence of simple processes. It follows that $g(t,q_n(t))\to g(t,q^Q(t))$, $dt\otimes d\W$-almost surely, since $g$ is continuous in the relative interior of its domain. 
Since the sequence $(q_n)$ is uniformly bounded, it follows from the assumption \eqref{TI: time integrability} that $\sup_ng(t,q_n(t)) \in L^1([0,1],dt)$.
By dominated convergence we then have
\begin{align}\label{pf:BBD-limit-11}\textstyle
\int_0^1g(t, q_n(t))\,dt \rightarrow \int_0^1g(t, q^Q(t))\,dt \quad P\text{-a.s.}
\end{align}
Now find $Q_n \in \Q$ such that $q^{Q_n}=q_n$, and note as before that $dQ_n/d\W \rightarrow dQ/d\W$ in $L^1(\W)$. The sequence $(\int_0^1g(t, q_n(t))\,dt)_n$ is essentially bounded thanks to \eqref{TI: time integrability}. Hence $E^{Q^n}[\int_0^1g(t,q_n(t))\,dt] \to E^{Q}[\int_0^1g(t,q^Q(t))\,dt]$. Since  $q^{Q_n}$ is a simple process, this proves  \eqref{eq:sub q simple}.

With \eqref{eq:sub q simple} in hand, we complete the proof as follows.	It is clear from the definition that {$\rho^g(H)$} is larger than the right-hand side of \eqref{eq:sub q SDE}. The reverse inequality follows from \eqref{eq:sub q simple} and the fact that whenever $q$ is a simple process in the sense described above, the SDE \eqref{eq:SDE_simpledrift} admits a unique strong solution.
\end{proof}

We can now provide the proof of Theorem \ref{thm BBD}. Our argument is reminiscent of \cite{Lehec}.

\begin{proof}[Proof of Theorem \ref{thm BBD}]
We prove Equation \eqref{rem BBD measures}, establishing first
the inequality ``$\le$''. By Lemma \ref{lem useful SDE}, we fix $Q \in \Q$ such that $q^Q \in \L_b^s$. 
Note that the completed filtrations of $W$ and $W^Q$ coincide, where $W^Q := W-\int_0^\cdot q^Q(t)dt$ is a $Q$-Brownian motion by Girsanov's theorem. Hence, there exists $\widetilde{q} \in \L_b$ such that $q^Q(t)=\tilde q(t,W^Q)$ and so $Q = P\circ (W+\int_0^\cdot \tilde q(t,W)dt)^{-1}$. Thus
\begin{align*}
\textstyle H(Q)-\E^Q\left[\int_0^1 g(t,q^Q(t,W))dt\right] &\textstyle =  H(Q)-\E^Q\left[\int_0^1 g(t,\widetilde{q}(t,W^Q))dt\right] \\ 
	&\textstyle = H\left (P\circ \left(W+\int_0^\cdot \tilde q(t,W)dt\right)^{-1} \right )-  \E\left[\int_0^1 g(t,\widetilde{q}(t,W))dt\right] \\
	&\textstyle  \le \sup_{q\in\mathcal L_b}\left\{H(Q^q)- \E\left [\int_0^1 g(t,q(t))dt\right ]\right\}.
\end{align*}

To prove the opposite inequality, let $q \in \L_b$, and set
\[
X(t) = W(t) + \int_0^tq(s)ds = W(t) + \int_0^tq(s,W)ds.
\]
Letting $\FF^X=(\F^X_t)_{t \in [0,1]}$ denote the complete filtration generated by $X$, let us choose $\widetilde{q} : [0,1] \times \C \rightarrow \R^d$ to be any bounded progressively measurable function satisfying
\begin{align*}
\widetilde{q}(t,X) = \E[q(t,W) \, | \, \F^X_t], \ \ a.s., \ \text{ for each } t \in [0,1].
\end{align*}
In particular, $\tilde{q}$ may be defined via optional projection. It is well known \cite[Exercise (5.15)]{Revuz1999} that the \emph{innovation process}
\[
\widetilde{W}(t) := X(t) - \int_0^t\widetilde{q}(s,X)ds
\]
is an $\FF^X$-Brownian motion. Hence, if $Q := \W \circ X^{-1}$, then $q^Q = \widetilde{q}$ by Girsanov's theorem. Using convexity of $g$ and Jensen's inequality, we conclude
\begin{align*}
\textstyle H\left ( P\circ \left( W+\int_0^\cdot q(t)dt \right)^{-1} \right )-\E\left[ \int_0^1g(t,q(t))dt\right] & \textstyle = H(Q)-\E\left[ \int_0^1g(t,q(t,W))dt\right] \\
	&\textstyle \le H(Q)-\E\left[\int_0^1g(t,\widetilde{q}(t,X))dt\right] \\
	&\textstyle = H(Q)-\E^Q\left[ \int_0^1g(t,\widetilde{q}(t,W))dt\right] \\
	&\le \rho^g(H),
\end{align*}
where the last inequality follows from the identity $q^Q = \widetilde{q}$ and the (overloaded) definition of $\rho^g$. As this inequality is valid for any $q \in \L_b$, the proof of Equation \eqref{rem BBD measures} is complete. Finally, Equation \eqref{eq BBD} follows since $Q\mapsto H(Q):=\E^Q[F(W)]$ is sequentially continuous in the desired way (if $F$ is bounded and Borel) and $\rho^g(H)=\rho^g(F)$ of course.
\end{proof}
}

The functional $\rho^g$ can be extended to random variables $X \in L^0(\W)$ that are bounded from below by setting $\rho^g(X):= \lim_{n\to \infty}\rho^g(X\wedge n)$. It is easily checked that this extension also satisfies \eqref{eq BBD}, though we will make no use of this.

Using Theorem \ref{thm BBD}, we now prove the Schilder-type result of Theorem \ref{co:Schilder-finaltime}. The argument is reminiscent of the weak convergence proof of the Freidlin-Wentzell theorem \cite[Theorem 4.3]{Boue-Dup}. \\

\noindent\textbf{Proof of Theorem \ref{co:Schilder-finaltime}}.
By \eqref{eq BBD} we have
\begin{align*}
\rho^{g_n}\left(F\left(\frac{W}{\sqrt{n}}\right) \right ) &\textstyle = \sup_{q \in \L_b} \E\left[ F \left ( \frac{W+\int_0^\cdot q(s) ds }{\sqrt{n}} \right) - \int_0^1 g\left(t,\frac{q(t)}{\sqrt{n}}\right)dt  \right ]\\
&\textstyle = \sup_{q \in \L_b} \E\left[ F \left ( \frac{W}{\sqrt{n}}  +\int_0^\cdot q(s) ds \right) - \int_0^1 g(t,q(t))dt  \right ].
\end{align*}
We first bound the $\liminf_{n\rightarrow\infty}$ of the above expression. For each absolutely continuous $\omega \in \C_0$ such that $\int_0^1 g(t,\dot \omega(t))dt<\infty$,  define the absolutely continuous path $w_k \in \C_0$ by setting
\begin{equation*}
\dot{w}_k(t):= \dot\omega(t)1_{\{|\dot\omega(t)|\le k \}} + \frac{k}{|\dot\omega(t)|}\dot\omega(t)1_{\{|\dot\omega(t)|>k\}},\quad k\ge 1.
\end{equation*}
Note that $w_k \in \L_b$.
For every $k\in \N$ we have
\begin{align}
\nonumber
	\liminf_{n\to \infty}\rho^{g^n}\left(F\left(\frac{W}{\sqrt{n}} \right) \right) &\ge \liminf_{n\to \infty}\E\left[ F \left ( \frac{W}{\sqrt{n}}  + w_k \right) - \int_0^1 g(t,\dot{w}_k(t))dt  \right ]\\
	\label{eq:first ineq schilder}
	& \ge F(w_k) - \int_0^1 g(t,\dot{w}_k(t))dt.
\end{align}
By convexity of $g(t,\cdot)$, we have
\begin{equation*}
 	g(t, \dot{w}_k(t)) \le g(t,\dot\omega(t)) + g(t,0) + 2b,
\end{equation*}
where $b\ge0$ is a constant such that $g\ge -b$. Moreover, since $\dot{w}_k(t) = \dot{\omega}(t)$ for sufficiently large $k$, it holds that $w_k \to \omega$ and $g(t, \dot{w}_k(t))\to g(t,\dot\omega(t))$ for every $t$. Thus, taking the limit as $k$ goes to infinity in \eqref{eq:first ineq schilder}, it follows by dominated convergence (noting that $|\dot{w}_k| \le |\dot\omega|$) that
 \begin{equation*}
 	\liminf_{n\to \infty}\rho^{g^n}\left(F\left(\frac{W}{\sqrt{n}} \right) \right) \ge F \left (\omega \right) - \int_0^1 g(t,\dot\omega(t))dt.
 \end{equation*}
Recalling the convention that $\int_0^1 g(t,\dot{\omega}(t))dt := \infty$ whenever $\omega$ is not absolutely continuous, we may take the supremum over $\omega \in \C_0$ to get
\begin{align*}
\liminf_{n\rightarrow\infty} \rho^{g_n}\left(F\left(\frac{W}{\sqrt{n}}\right) \right ) &\ge \sup_{\omega \in \C_0} \left( F \left ( \omega \right) - \int_0^1 g(t,\dot{\omega}(t))dt  \right).
\end{align*}

For the opposite inequality, first notice that we may always choose a constant $q \equiv 0$ to get the lower bound
\begin{align}
\rho^{g_n}\left(F\left(\frac{W}{\sqrt{n}}\right) \right ) \ge \E\left[ F \left ( \frac{W}{\sqrt{n}}\right) - \int_0^1g(t,0)dt\right] \ge -2C, \ \ \forall n \in \N, \label{pf:schilder-lowerbound}
\end{align}
where $C < \infty$ is any constant such that $\inf_{\omega \in \C}F(\omega) \ge -C$ and $\int_0^1g(t,0)dt \le C$ (see Assumption (TI)).
Now, take $q_n$ to be $1/n$-optimal; that is, let $q_n \in \L_b$ be such  that
\begin{align}
\rho^{g_n}\left(F\left(\frac{W}{\sqrt{n}}\right) \right ) - \frac{1}{n} \le 
\E\left[ F \left ( \frac{W}{\sqrt{n}}  +\int_0^\cdot q_n(s) ds \right) - \int_0^1 g(t,q_n(t))dt  \right ]. \label{pf:schilder-UB}
\end{align}
From \eqref{pf:schilder-lowerbound}, we have
\begin{align}
\sup_n\E\int_0^1 g(t,q_n(t))dt < \infty. \label{pf:schilder-expbound}
\end{align}
Letting $A_n(t) = \int_0^tq_n(s)ds$, it follows from Lemma \ref{le:H1-tightness} that the sequence $(A_n)$ of $\C_0$-valued random variables is tight. Moreover, if we fix a  subsequence $A_{n_k}$ which converges in law to some $A$, then we may write $A = \int_0^\cdot q(t)dt$ for some process $q$ satisfying
\[
\E\int_0^1 g(t,q(t))dt \le \liminf_{k\rightarrow\infty}\E\int_0^1 g(t,q_{n_k}(t))dt.
\]
Because $\lim_{n\rightarrow\infty} W/\sqrt{n} = 0$ in probability, we have $W/\sqrt{n_k} + A_{n_k} \rightarrow A$ in law.
Recalling \eqref{pf:schilder-UB}, we have (taking limits still along the same subsequence)
\begin{align*}
\limsup_{k\rightarrow\infty} \rho^{g_{n_k}}\left(F\left(\frac{W}{\sqrt{n_k}}\right) \right ) &\le \limsup_{k\rightarrow\infty} \E\left[ F\left( \frac{W}{\sqrt{n_k}} + A_{n_k} \right) - \int_0^1 g(t,q_{n_k}(t))dt  \right ] \\
	&\le \E\left[ F(A) - \int_0^1 g(t,q(t))dt  \right ] \\
	&= \E\left[ F\left( \int_0^\cdot q(t)dt \right) - \int_0^1 g(t,q(t))dt  \right ] \\
	&\le \sup_{\omega \in \C_0} \left( F \left ( \omega \right) - \int_0^1 g(t,\dot{\omega}(t))dt  \right).
\end{align*}
We have argued that for any subsequence we can extract a further subsequence along which the above limsup bound is valid, and we conclude that the same upper bound is valid without passing to a subsequence.
This completes the proof.\hfill\qedsymbol

\section{The Sanov-type limit theorem} \label{sec: proof Limit Thms}

This section develops the necessary machinery for proving Theorem \ref{th:mainlimit},  some of which will be used again in later sections. The goal is to write our problem in a setting amenable to \cite[Theorem 1.1]{Lac-Sanov}. 
A first key step is to use Theorem \ref{thm BBD} to derive an alternative expression for the pre-limit quantity in Theorem \ref{th:mainlimit}, relating it to the iterates denoted $\rho_n$ in \cite{Lac-Sanov}, and this will explain the precise form of the scaling limit. This is carried out in Section \ref{sec: stoch rep rho n}. A second key ingredient in applying \cite{Lac-Sanov} is to check that the sub-level sets of $\alpha^g$ are weakly compact, which turns out to fail in general. Section \ref{se:compactness-alpha} provides a suitable work-around. Finally, Section \ref{se:proof of mainlimit} assembles these pieces into a complete proof.

\subsection{The rescaled control problem}
\label{sec: stoch rep rho n}

Let $\C^n$ be the $n$-fold product space, and denote by $(\omega_1,\dots,\omega_n)$ a typical element in $\C^n$. Let $B_b(\C^n)$ be the space of bounded measurable functions on $\C^n$. We define inductively the iterates of $\rho^g_n : B_b(\C^n) \rightarrow \R\cup\{+\infty\}$  as follows: We set $\rho_1^g \equiv \rho^g$, and for $n > 1$ define
\begin{align}
\label{eq iterates g}
\rho_n^g(f) := \rho_{n-1}^g\bigl(\,(\omega_1,\ldots,\omega_{n-1}) \mapsto \rho(f(\omega_1,\ldots,\omega_{n-1},\cdot))\,\bigr).
\end{align}
In other words, given $f \in B_b(\C^n)$ for $n > 1$, we define\footnote{Actually, the function $\widetilde{f}$ is merely upper-semianalytic in general. But this does not pose any problems, since upper-semianalytic functions are universally measurable.} $\widetilde{f} \in B_b(\C^{n-1})$ by $\widetilde{f}(\omega_1,\ldots,\omega_{n-1}) = \rho^g(f(\omega_1,\ldots,\omega_{n-1},\cdot))$,
and then we set $\rho_n^g(f) = \rho_{n-1}^g(\widetilde{f})$.

Recall from \eqref{eq:chopped paths} the definition of the chopped paths $W_{(n,k)}$ for $k=1,\ldots,n$.
The following representation for $\rho^g_n$ underlies our proof of Theorem \ref{th:mainlimit}:

\begin{proposition} \label{pr:rhon}
For $q \in \L_b$ define $X^q = W + \int_0^\cdot q(t)dt$.
For $f \in B_b(\C^n)$, we have
\begin{align*}
\rho^g_n(f) &= \sup_{q \in \L_b} \E\left[f\left(X^q_{(n,1)},\ldots,X^q_{(n,n)}\right) - n\int_0^1g\left(nt-\lfloor nt \rfloor,\frac{q(t)}{\sqrt{n}}\right)dt \right] \\
	&= n\rho^{G_n}\left(\frac{1}{n}f(W_{(n,1)},\ldots,W_{(n,n)})\right).
\end{align*}
\end{proposition}
\begin{proof}
The second claimed equality follows immediately from Theorem \ref{thm BBD} and the definition of $G_n$, so we prove only the first.

For $n=1$ this is Theorem \ref{thm BBD}. Fix $n > 1$. Define a process $B^n : [0,n] \times \C^n \rightarrow \R^d$ by setting
\[
B^n(t,\omega_1,\ldots,\omega_n) = \begin{cases}
\omega_1(t) - \omega_1(0) &\text{if } t \in [0,1] \\
\omega_{k+1}(t-k) - \omega_{k+1}(0) + {\sum_{i=1}^k[\omega_i(1)-\omega_i(0)]} &\text{if } t \in [k,k+1], \ k \leq n-1.
\end{cases}
\]
In other words, $B^n(t,\omega_1,\ldots,\omega_n)$ follows the increments of $\omega_k$ on the interval $[k-1,k]$. Define the filtration $\FF^n$ on $\C^n$ by setting $\F^n_t = \sigma(B^n_s : s \le t)$. Note that $B^n=(B^n(t))_{t \in [0,n]}$ is a Brownian motion on $(\C^n,\FF^n,\W^n)$ {with $\W^n$ the $n$-fold product of $\W$}. In the following, the symbol $\E^n$ will denote expectation on $(\C^n,\FF^n,\W^n)$, and we note that $\E=\E^1$.

Let $\A_n$ denote the set of bounded $\FF^n$-progressively measurable processes $q : [0,n] \times \C^n \rightarrow \R^d$.
For $q \in \A_n$, define a continuous process $X^{n,q}=(X^{n,q}(t))_{t \in [0,n]}$ on $(\C^n,\FF^n,\W^n)$ by
\[
X^{n,q}(t,\omega_1,\ldots,\omega_n) := \int_0^tq(s,\omega_1,\ldots,\omega_n)ds + B^n(t,\omega_1,\ldots,\omega_n).
\]
In the following, for a path $x \in C([0,n];\R^d)$ and for $k=1,\ldots,n$, define the chopped (but not rescaled) path $x_{(c,n,k)} \in \C=C([0,1];\R^d)$ by 
\[
x_{(c,n,k)}(t) = x(k-1+t) - x(k-1), \quad t \in [0,1].
\]
In other words, $x_{(c,n,k)}$ is simply the increment over the time interval $[k-1,k]$.

Let us understand first the case $n=2$. For a fixed $\omega \in \C$, by Theorem \ref{thm BBD} we have
\[
\rho^g(f(\omega,\cdot))  = \sup_{q \in \A_1}\E^1\left[f(\omega,X^{1,q}) - \int_0^1g(t,q(t))dt\right].
\]
Applying Theorem \ref{thm BBD} once again, we have by definition
\begin{align*}
\rho_2^g(f) &= \sup_{\beta \in \A_1}\E^1\left[\rho^g(f(X^{1,\beta},\cdot)) - \int_0^1g(t,\beta(t))dt\right] \\
	&= \sup_{\beta \in \A_1}\E^1\left[\left(\sup_{q \in \A_1}\E\left[f(\omega,X^{1,q}) - \int_0^1g(t,q(t))dt\right]\right)\Bigg|_{\omega=X^{1,\beta}} - \int_0^1g(t,\beta(t))dt\right].
\end{align*}
The key idea here is to apply a form of dynamic programming. In particular, let $\widehat\A_1$ denote the set of functions $[0,1] \times \C \times \C \ni (t,\omega_1,\omega_2) \mapsto \widehat q[\omega_1](t,\omega_2) \in \R^d$ which are jointly measurable, using the progressive $\sigma$-field on $[0,1] \times \C$ for the argument $(t,\omega_2)$ and the Borel $\sigma$-field on $\C$ for the argument $\omega_1$. A standard measurable selection argument \cite[Proposition 7.50]{BertsekasShreve} lets us write the above as
\begin{align}
\rho^g_2(f) &= \sup_{\stackrel{\beta \in \A_1}{\widehat q \in \widehat\A_1}}\E^1\left[\E^1\left[f(\omega_1,X^{1,\widehat{q}[\omega_1]}) - \int_0^1g(t,\widehat{q}[\omega_1](t))dt\right]\Bigg|_{\omega_1=X^{1,\beta}} \!\!- \int_0^1g(t,\beta(t))dt\right]. \label{pf:iterate1}
\end{align}
Now consider a fixed $\widehat q \in \widehat\A_1$ and $\beta \in \A_1$. We may define a process $q : [0,2] \times \C^2 \rightarrow \R^d$ by setting
\begin{align}
q(t,\omega_1,\omega_2) = \beta(t,\omega_1)1_{[0,1]}(t) + \widehat q[\omega_1](t-1,\omega_2)1_{(1,2]}(t). \label{pf:iterate2}
\end{align}
Then $q \in \A_2$, and unpacking the definitions reveals the identities
\begin{align*}
X^{1,\widehat{q}[\omega_1]}(t,\omega_2) &= X^{2,q}_{(c,2,2)}(t,\omega_1,\omega_2), \ \ t \in [0,1], \\
X^{1,\beta}(t,\omega_1) &= X^{2,q}_{(c,2,1)}(t,\omega_1,\omega_2), \ \ t \in [0,1], \\
\E^1\int_0^1g(t,\beta(t))dt &= \E^2\int_0^1g(t,q(t))dt,
\end{align*}
which in turn imply
\begin{align*}
&\E^1\left[f(\omega_1,X^{1,\widehat{q}[\omega_1]}) - \int_0^1g(t,\widehat{q}[\omega_1](t))dt\right]\Bigg|_{\omega_1=X^{1,\beta}} \\
	&\quad =\E^2\left[f(\omega_1,X^{2,q}_{(c,2,2)}(\omega_1,\cdot)) - \int_1^2g(t-1,q(t,\omega_1,\cdot))dt\right]\Bigg|_{\omega_1=X^{1,\beta}} \\
	&\quad = \E^2\left[\left. f(X^{2,q}_{(c,2,1)},X^{2,q}_{(c,2,2)}) - \int_1^2g(t-1,q(t))dt \, \right| \, \F^2_1 \right].
\end{align*}
Indeed, the last identity follows from the fact that the $\C$-valued random variable $(\omega_1,\omega_2) \mapsto \omega_2$ is independent of $\F^2_1$.
Finally, we plug this last expression into \eqref{pf:iterate1}.
Then, note that the map $(\beta,\widehat q) \mapsto q$ given by \eqref{pf:iterate2} defines a bijection between $\A_1 \times \widehat\A_1$ and $\A_2$, and use the tower property of conditional expectation to get
\begin{align*}
\rho_2^g(f) &= \sup_{q \in \A_2}\E^2\left[f(X^{2,q}_{(c,2,1)},X^{2,q}_{(c,2,2)}) - \int_0^2g(t - \lfloor t \rfloor, q(t))dt\right].
\end{align*}
This argument adapts, mutatis mutandis, to the case of general $n > 1$, and we find
\begin{align}
\rho_n^g(f) = \sup_{q \in \A_n}\E^n\left[f(X^{n,q}_{(c,n,1)},\ldots,X^{n,q}_{(c,n,n)}) - \int_0^ng(t-\lfloor t \rfloor,q(t))dt\right]. \label{pf:iterate3}
\end{align}

To complete the proof, we rescale this control problem to live on the time interval $[0,1]$ instead of $[0,n]$. Still working on the space $(\C^n,\FF^n,\W^n)$, define for each $q \in \A_n$ the process
\[
\overline{X}^{n,q}(t) := \frac{1}{\sqrt{n}}X^{n,q}(nt) = \frac{1}{\sqrt{n}}\int_0^{nt}q(s)ds + \frac{1}{\sqrt{n}}B^n({nt}), \quad \text{ for } t \in [0,1].
\]
By a change of variables and Brownian scaling, we can write
\[
\overline{X}^{n,q}(t) = \int_0^t\overline{q}(s)ds + \overline{B}^n(t),
\]
where $\overline{q}(s) := \sqrt{n}q(ns)$, and $\overline{B}^n(t) := \frac{1}{\sqrt{n}}B^n(nt)$ is a Brownian motion. Another change of variables yields 
\[
\int_0^ng(t-\lfloor t \rfloor,q(t))dt = n\int_0^1g\left(nt-\lfloor nt \rfloor,\frac{\overline{q}(t)}{\sqrt{n}}\right)dt.
\]
Lastly, it is straightforward to check that $X^{n,q}_{(c,n,k)} \equiv \overline{X}^{n,q}_{(n,k)}$. Putting it all together, \eqref{pf:iterate3} becomes
\begin{align*}
\rho_n^g(f) = \sup_{q \in \A_n}\E^n\left[f(\overline{X}^{n,q}_{(n,1)},\ldots,\overline{X}^{n,q}_{(n,n)}) - n\int_0^1g\left(nt-\lfloor nt \rfloor,\frac{\overline{q}(t)}{\sqrt{n}}\right)dt\right].
\end{align*}
Complete the proof by transferring everything from the probability space $(\C^n,\FF^n,\W^n)$ to the original space $(\C,\FF,\W)$, using the map $\C^n \ni (\omega_1,\ldots,\omega_n) \mapsto \overline{B}^n(\omega_1,\ldots,\omega_n) \in \C$.
\end{proof}

\subsection{In search of compactness} \label{se:compactness-alpha}
As mentioned above, the goal of this section is to overcome the technical impediment that the functional $\alpha^g$ does not necessarily have compact sub-level sets. We illustrate this with an example, but we stress that this is only an issue when we do not assume that $g$ has at least quadratic growth. 

\begin{example}
Take $d=1$ and $g(t,q)=|q|^{5/4}$. Set $q_n(t):=t^{-3/4}{\bf 1}_{(1/n,1]}(t)$ and $q_{\infty}(t):= t^{-3/4}$. Define $Q_n$ as the measure with density
\[
\frac{dQ_n}{d\W} = \exp\left (\int_0^1 q_n(t) dW(t) - \frac{1}{2}\int_0^1 |q_n(t)|^2dt \right ),
\]
and let $Q_\infty = \text{Law}(W+\int_0^\cdot q_\infty(t) dt)$. One can easily check the following:
\begin{enumerate}
\item $Q_n$ converges to $Q_\infty$ in the weak topology of measures.
\item $\alpha^g(Q_n)\leq 16$ for each $n$.
\item $Q_\infty$ is singular to $\W$, so in particular $\alpha^g(Q_\infty)= \infty$.
\end{enumerate}
This shows that the sublevel set $\{\alpha^g\leq 16\}$ is not even closed in the weak topology of measures.
\end{example}

For this reason,  we initially replace $\alpha^g$ and $\rho^g$ by two new functionals better suited for our purposes.  Let ${\cal P}^*$ denote the set of those measures $Q$ on $\C$ for which {there exists a} progressive $\R^d$-valued process $q^Q$ {such that $\int_0^1|q^Q(s)|ds < \infty$ $Q$-a.s.\  and}
$$ W(t) - \int_0^t q^Q(s)ds \,\,\,\text{ is a $Q$-Brownian motion}.$$ 
{The process $q^Q$ is then uniquely defined up to $dt \otimes dQ$-almost everywhere equality.}
This does not reduce to Girsanov theory, since we are not asking that elements in ${\cal P}^*$ be absolutely continuous with respect to Wiener measure (e.g.\ the set ${\cal P}^*$ contains measures singular to $\W$, such as the laws of Brownian bridges or Bessel processes). Note, however, that $\alpha^g(Q)=\tilde{\alpha}^g(Q)$ for $Q \in \Q$.

Consider the functional
\begin{align}
{\cal P}^*\ni Q\mapsto \tilde{\alpha}^g(Q):= \E^Q\left[\int_0^1 g(t,q^Q(t))dt\right]\in \R\cup \{+\infty\}, \label{def:tildealpha}
\end{align}
where we define the functional as $+\infty$ outside of ${\cal P}^*$.
Let $B_b(\C)$ denote the set of bounded measurable functions on $\C$ and define the functional
$$\tilde{\rho}^g(F):=\sup_{Q\in {\cal P}^*} (\E^Q[F]-\tilde{\alpha}^g(Q)), \quad F \in B_b(\C).$$
We now give some elementary facts about $\tilde{\alpha}^g$ which may seem folklore. We defer the rather technical proof of the next lemma to Appendix \ref{sec pending proofs}. Recall that we are assuming at all times that the given function $g$ satisfies assumption (TI).

\begin{lemma}\label{thm technical}
The functional $\tilde{\alpha}^g$ is convex, lower semicontinuous with respect to weak convergence of measures on path space, and its sub-level sets are weakly compact in this topology. Furthermore, we have
$$\tilde{\alpha}^g(Q)=\sup_{F\in B_b(\C)}\left(E^Q[F]-\tilde{\rho}^g(F)\right)={\sup_{F\in C_b(\C)}\left(E^Q[F]-\tilde{\rho}^g(F)\right),\,\,\,\, Q\in \mathcal{P}^*}.$$
\end{lemma}

In general $\rho^g$ and $\tilde{\rho}^g$, just as $\alpha^g$ and $\tilde{\alpha}^g$, may differ. It is thus important to establish how $\rho^g$ and $\tilde{\rho}^g$ are related. This is the content of the next result.

\begin{lemma}\label{lem concidence}
If $F : \C \to \R$ is bounded and lower semicontinuous, then $\rho^g(F)=\tilde{\rho}^g(F)$.
\end{lemma}
\begin{proof}
Obviously $\tilde{\rho}^g\geq \rho^{g}$. Let $Q\in \mathcal P^*$ such that $\tilde{\alpha}^g(Q)<\infty$. We will exhibit a sequence $Q^n$ of absolutely continuous measures such that $$\liminf \{\E^{Q_n}[F]-\alpha^{g}(Q_n)\}\geq \E^{Q}[F]-\tilde{\alpha}^g(Q) ,$$ which would establish the claim. Note that $\E^Q\int_0^1g(t,q^Q(t))dt=\tilde{\alpha}^g(Q) < \infty$ implies $q^Q(t) \in \mathrm{dom}(g(t,\cdot))$, $dt\otimes d\W$-a.e.
We know that $W^Q(t) := W(t) - \int_0^tq^Q(s)ds$ is a $Q$-Brownian motion. Define $q_n(t) = q^Q(t)1_{\{|q^Q(t)| \le n\}}$, and let $Q_n$ denote the law of the process
\[
X^n(t) = W^Q(t) + \int_0^tq_n(s)ds.
\]
Note that $q_n$ is uniformly bounded, and so $Q_n \in \Q$.
Because $q_n(t) \rightarrow q^Q(t)$ for each $t$, it is clear that $Q_n \rightarrow Q$ weakly. Hence, by lower semicontinuity of $F$,
\begin{align*}
\liminf_{n\rightarrow\infty}\E^{Q_n}[F] \ge \E^Q[F].
\end{align*}
Finally, define $\widehat{q}_n(t,X^n)$ and $\widehat{W}^Q$ as the optional projections (under $Q$) of $q_n$ and $W^Q$, respectively, on the filtration generated by $X^n$. Then $\widehat{W}^Q$ remains a Brownian motion in this smaller filtration \cite[Exercise (5.15)]{Revuz1999}.
It follows that $q^{Q_n}(t,X^n) = \widehat{q}_n(t,X^n)$, $dt\otimes dQ$-almost surely. By convexity, we get
\begin{align*}
\alpha^g(Q_n) &= \E^{Q_n}\left[ \int_0^1 g(t,q^{Q_n}(t,W))dt \right] = \E^Q\left[ \int_0^1 g(t,\widehat{q}_n(t,X^n))dt \right] \le \E^Q\left[ \int_0^1 g(t,q_n(t))dt \right] \\
	&= \E^Q\left[ \int_0^1 \left(g(t,q(t))1_{\{|q^Q(t)| \le n\}} + g(t,0)1_{\{|q^Q(t)| > n\}}\right)dt \right].
\end{align*}
Since $\int_0^1 g(t,0)dt < \infty$ by assumption (TI), we conclude from monotone convergence that 
\begin{align}\label{eq approx by abs cont}
\limsup_n \alpha^g(Q_n) \leq \E^Q\left[\int_0^1 g(t,q(t))dt\right] = \tilde{\alpha}^g(Q).
\end{align}
\end{proof}

Recalling the definition of the iterates $\rho^g_n$ based on $\rho^g$ and given in \eqref{eq iterates g}, we define the iterates $\tilde{\rho}^g_n$ based on $\tilde{\rho}^g$ in the same way. A simple consequence of Lemma \ref{lem concidence} is that $\rho^g_n=\tilde{\rho}^g_n$ restricted to a large class of functions: 

\begin{lemma}\label{lem coincidence n}
Let $n \in \mathbb{N}$, and let $f : \C^n\to\R$ be lower semicontinuous and bounded. Then the functions  $\C^{n-1}\ni(\omega_1,\dots,\omega_{n-1})\mapsto \rho(f(\omega_1,\dots,\omega_{n-1},\cdot))$ are lower-semicontinuous and bounded, for both $\rho=\rho^g$ and $\rho=\tilde{\rho}^g$. In particular, for such $f$ we have $\rho^g_n(f)=\tilde{\rho}^g_n(f)$. 
\end{lemma}
\begin{proof}
The case $n=1$ is covered by Lemma \ref{lem concidence}. The general case follows by induction but for ease of presentation we consider only the case $n=2$. Let us prove that $\omega\mapsto F(\omega):=\tilde{\rho}^g(f(\omega, \cdot))$ is lower semicontinuous. To wit, if $\omega_n\to \omega $ and $F(\omega_n)\leq c$ for all $n$, then by definition $$\textstyle\int f(\omega_n,\bar{\omega})dQ(\bar{\omega})-\tilde{\alpha}^g(Q)\leq c,$$ for all $Q\in{\cal P}^*$. Taking limit inferior here, and by Fatou's lemma and lower semicontinuity of $f$, we get $$\textstyle\int f(\omega,\bar{\omega})dQ(\bar{\omega})-\tilde{\alpha}^g(Q)\leq c.$$
Now taking supremum over $Q$ we conclude $F(\omega)\leq c$.
Moreover, because $g$ is bounded from below and $f$ is bounded, $F$ too is bounded. The same reasoning can be applied to $\rho^g$. By Lemma \ref{lem concidence} and the case $n=1$ we have $$\rho^g_2(f)=\rho^g(\,\omega\mapsto \rho^g(f(\omega,\cdot))\,)=\tilde{\rho}^g(\,\omega\mapsto \rho^g(f(\omega,\cdot))\,)=\tilde{\rho}^g(\,\omega\mapsto \tilde{\rho}^g(f(\omega,\cdot))\,)=\tilde{\rho}^g_2(f).$$

\end{proof}

\subsection{Proof of Theorem \ref{th:mainlimit}} \label{se:proof of mainlimit}
With the above machinery we can finally prove Theorem \ref{th:mainlimit}.
Let us denote the empirical measure of the family $(\omega_1,\ldots,\omega_n) \in \C^n$ by
\begin{equation*}
L_n(\omega_1,\ldots,\omega_n) := \frac 1n \sum_{i\le n}\delta_{\omega_i} \label{eq L}
\end{equation*} 
and recall the notation $\omega_{(n,k)}$ from \eqref{eq:chopped paths}. Apply Proposition \ref{pr:rhon} to get
\begin{align*}
\rho^{G_n}\left(F \left(\frac 1n \sum_{k=1}^n\delta_{W_{(n,k)}}\right)\right) &= \frac{1}{n}\rho^g_n(nF \circ L_n).
\end{align*}
Since $F \circ L_n$ is clearly a continuous function on $\C^n$, Lemma \ref{lem coincidence n} yields $\rho^g_n(nF \circ L_n) = \tilde{\rho}^g_n(nF \circ L_n)$. Now, because $\tilde{\alpha}^g$ is convex and has weakly compact sub-level sets, we may apply \cite[Theorem 1.1]{Lac-Sanov} (taking note of the representation of \cite[Proposition A.1]{Lac-Sanov}) to get
\begin{align*}
\lim_{n\rightarrow\infty}\frac{1}{n}\tilde{\rho}^g_n(nF \circ L_n) = \sup_{Q \in \P(\C)}(F(Q) - \tilde{\alpha}^g(Q)) = \sup_{Q \in \P^*}(F(Q) - \tilde{\alpha}^g(Q)).
\end{align*}
To complete the proof, it remains to show that
\begin{align}
\sup_{Q \in {\cal P}^*}\left(F(Q) - \tilde{\alpha}^g(Q)\right) &= \sup_{Q \in \Q}(F(Q)-\alpha^g(Q)). \label{pf:mainlimit1}
\end{align}
Indeed, this will prove the first equality of Theorem \ref{th:mainlimit}, while the second follows from Theorem \ref{thm BBD}.
To prove \eqref{pf:mainlimit1}, notice from the proof of  Lemma \ref{lem concidence} (specifically \eqref{eq approx by abs cont}) that the following holds: If $\tilde{\alpha}^g(Q)<\infty$, then there exist $Q_n \in \Q$ such that $Q_n \to Q$ weakly and $\limsup_n\alpha^g(Q_n)\leq \tilde{\alpha}^g(Q)$. From this and continuity of $F$ we deduce \eqref{pf:mainlimit1}.  \hfill\qedsymbol

\section{BSDE scaling limits} \label{sec BSDE scaled limits}
This section is dedicated to the proofs of Theorems \ref{thm HL} and \ref{thm: Sanov time}.
We will make use of the following definitions. For a function $g$ satisfying (TI) and for $t \in [0,1)$, define $g^{(t)} : [0,1] \times \R^d \rightarrow \R \cup \{\infty\}$ by
\begin{align*}
g^{(t)}(s,q) := (1-t) g\left(t + s(1-t), \frac{q}{\sqrt{1-t}} \right).
\end{align*}
Note that $g^{(t)}$ itself satisfies (TI), and so $\rho^{g^{(t)}}$ is well defined.
Moreover, we define the operation $\otimes_t : \C\times \C_0 \rightarrow \C$ by
\[
\omega \otimes_t \overline\omega (s) := \omega(s \wedge t) + \sqrt{1-t}\,\overline\omega\left(\frac{s-t}{1-t}\right)1_{[t,1]}(s).
\]
We begin with the following crucial lemma, which shows how to express the (super-) solution process $Y(t)$ of a BSDE with generator $g^*$ in terms of $\rho^{g^{(t)}}$.

\begin{lemma} \label{le:BSDE-timedep}
Let $F \in C_b(\C)$, and let $(Y,Z)$ be the minimal supersolution of
\[
dY(t) = - g^*(t,Z(t)) dt + Z(t) dW(t), \quad Y(1) = F(W).
\]
Then, for $t \in [0,1)$ and $\W$-a.e.\ $\omega \in \C$, we have
\begin{align*}
Y(t,\omega) = \rho^{g^{(t)}}(F(\omega \otimes_t \cdot)).
\end{align*}
\end{lemma}
\begin{proof}
By Lemma \ref{lem:rep Qt} {(which is just a minor modification of \cite[Theorem 3.4]{tarpodual})}, it holds
\[
Y(t) = \esssup_{Q \in \Q_t}\E^{Q}\left[ F(W) - \int_t^1 g\left(s,q^Q(s,W)\right)ds \, \Big| \, \F_t\right],
\]
where $\Q_t$ is the set of those measures $Q\in\Q$ such that $Q=\W$ on $\mathcal F_t$. Note that $q^Q \equiv 0$ on $[0,t]$ for $Q \in \Q_t$. Now, for a path $\omega \in \C$, define $\omega^{(t)} \in \C_0$ by
\[
\omega^{(t)}(s) := \frac{1}{\sqrt{1-t}}\left(\omega(t + s(1-t)) - \omega(t)\right).
\]
It is readily checked that $\omega \otimes_t \omega^{(t)} = \omega$ for $\omega \in \C$. Hence, for a.e.\ $\omega$, we may write
\begin{align}
Y(t,\omega) = \esssup_{Q \in \Q_t}\E^{Q}\left[ F(\omega \otimes_t W^{(t)}) - \int_t^1 g\left(s,q^Q(s,\omega \otimes_t W^{(t)})\right)ds \, \Big| \, \F_t\right](\omega). \label{pf:Y exp 1}
\end{align}
On the other hand, we can write
\begin{align}
\rho^{g^{(t)}}(F(\omega \otimes_t \cdot)) &= \sup_{Q \in \Q}\E^Q\left[ F(\omega \otimes_t W) - \int_0^1g^{(t)}(s,q^Q(s,W))ds\right]. \label{pf:otimes-omega1}
\end{align}

With these preparations out of the way, we first show that
\begin{align}
Y(t,\omega) \le \rho^{g^{(t)}}(F(\omega \otimes_t \cdot)), \quad a.e. \ \omega. \label{pf:timedep-ineq1}
\end{align}
To see this, fix $Q \in \Q_t$. Define a measurable map $\C \ni \omega \mapsto Q_\omega \in \P(\C)$ as a version of $Q(W^{(t)} \in \cdot \, | \, \F_t)(\omega)$. Recalling also that $q^Q=0$ on $[0,t]$, and noting that $\W(W^{(t)} \in \cdot \, | \, \F_t) = \W$ a.s.\ by Brownian scaling, we have
\begin{align*}
\frac{dQ_\omega}{d\W}(\omega^{(t)}) &= \frac{dQ}{d\W}(\omega) = \exp\left(\int_t^1q^Q(s,\omega)d\omega(s) - \frac12 \int_t^1|q^Q(s,\omega)|^2ds\right), \quad \omega \in \C.
\end{align*}
Now, for $\omega \in \C$ define $\widetilde{q}_\omega : [0,1] \times \C_0 \to \R^d$ by
\[
\widetilde{q}_\omega(s,\overline\omega) := \sqrt{1-t} q^Q(t+s(1-t),\omega \otimes_t \overline\omega).
\]
Recalling that $\omega \otimes_t \omega^{(t)} = \omega$, by a change of variables we may write the above as
\begin{align*}
\frac{dQ_\omega}{d\W}(\omega^{(t)}) &= \exp\left(\int_0^1\widetilde{q}_\omega(s,\omega^{(t)})d\omega^{(t)}(s) - \frac12 \int_0^1|\widetilde{q}_\omega(s,\omega^{(t)})|^2ds\right).
\end{align*}
We conclude that $Q_\omega \in \Q$ and $\widetilde{q}_\omega = q^{Q_\omega}$ for a.e.\ $\omega$.  With these identifications and another change of variables in the time-integral, we can write
\begin{align*}
\E^{Q} &\left[ F(\omega \otimes_t W^{(t)}) - \int_t^1 g\left(s,q^Q(s,\omega \otimes_t W^{(t)})\right)ds \, \Big| \, \F_t\right](\omega) \\
	&= \E^{Q_\omega}\left[ F(\omega \otimes_t W) - \int_t^1 g\left(s,q^Q(s,\omega \otimes_t W)\right)ds\right] \\
	&= \E^{Q_\omega}\left[ F(\omega \otimes_t W) - (1-t)\int_0^1 g\left(t+s(1-t),q^Q(t+s(1-t),\omega \otimes_t W)\right)ds\right] \\
	&= \E^{Q_\omega}\left[ F(\omega \otimes_t W) - \int_0^1 g^{(t)}(s,q^{Q_\omega}(s,W))ds\right],
\end{align*}
with the last line simply using the definition of $g^{(t)}$. This completes the proof of \eqref{pf:timedep-ineq1}.

Finally, we prove the reverse, namely that 
\begin{align}
Y(t,\omega) \ge \rho^{g^{(t)}}(F(\omega \otimes_t \cdot)). \label{pf:timedep-ineq2}
\end{align}
First, note that the definition of the operation $\otimes_t$ entails that, for each $Q$, the function of $\omega$ on the right-hand side of \eqref{pf:otimes-omega1} is $\F_t$-measurable.
Using \cite[Proposition 7.50]{BertsekasShreve}, we may find an $\F_t$-measurable map $\C \ni \omega \mapsto Q_\omega \in \Q$ such that
\begin{align}
\E^{Q_\omega}\left[ F(\omega \otimes_t W) - \int_0^1g^{(t)}(s,q^{Q_\omega}(s,W))ds\right] \ge \rho^{g^{(t)}}(F(\omega \otimes_t \cdot)) - \epsilon, \label{pf:rho gt eps bound}
\end{align}
for each $\omega \in \C$. Define $Q$ by setting 
\[
\frac{dQ}{d\W}(\omega) = \frac{dQ_\omega}{d\W}(\omega^{(t)}).
\]
The $\F_t$-measurability of $\omega \mapsto Q_\omega$ and the independence of $W^{(t)}$ and $\F_t$ under $\W$ together ensure that $\W(d\omega)$ indeed integrates the right-hand side to $1$, so that $Q \in \P(\C)$ is well defined. Using the same facts, it is straightforward to check that $Q \in \Q_t$; indeed, if $S \in \F_t$ then
\begin{align*}
Q(S) &= \E\left[\frac{dQ}{d\W}1_S(W)\right] = \E\left[\frac{dQ_W}{d\W}(W^{(t)})1_S(W)\right] \\
	&= \int_\C \E\left[\frac{dQ_\omega}{d\W}(W^{(t)})\right] 1_S(\omega)\W(d\omega)  = \int_\C 1_S(\omega)\W(d\omega)  = \W(S).
\end{align*}
As argued in the previous paragraph, $\omega \mapsto Q_\omega$ is a version of $Q(W^{(t)} \in \cdot \, | \, \F_t)(\omega)$, and we have
\begin{align*}
q^{Q_\omega}(s,\overline\omega) = \sqrt{1-t}q^Q(t+s(1-t),\omega \otimes_t \overline\omega),
\end{align*}
for $\omega,\overline\omega \in \C$. 
Using \eqref{pf:rho gt eps bound}, the definition of $g^{(t)}$, and a change of variables, we find
\begin{align*}
&\rho^{g^{(t)}}(F(\omega \otimes_t \cdot)) \\ 
	&\le \epsilon + \E^{Q}\left[F(\omega \otimes_t W^{(t)}) - \int_0^1g^{(t)}(s,q^{Q_\omega}(s,W^{(t)}))ds \, \Big| \, \F_t\right](\omega) \\
	&=  \epsilon + \E^{Q}\left[F(\omega \otimes_t W^{(t)}) - (1-t)\int_0^1g(t+s(1-t),q^{Q}(t+s(1-t),\omega \otimes_t W^{(t)}))ds \, \Big| \, \F_t\right](\omega) \\
	&= \epsilon + \E^{Q}\left[F(\omega \otimes_t W^{(t)}) - \int_t^1g(s,q^{Q}(s,\omega \otimes_t W^{(t)}))ds \, \Big| \, \F_t\right](\omega).
\end{align*}
Comparing this to the expression \eqref{pf:Y exp 1}, the proof of \eqref{pf:timedep-ineq2} is complete.
\end{proof}

We now give the proof of Theorem \ref{thm HL}.
In the following, define $\C_0[t,1]$ to be the set of continuous paths $\omega : [t,1] \to \R$ with $\omega(t)=0$. For $\omega \in \C$ and $\overline\omega \in \C_0[t,1]$, define $\omega \oplus_t \overline\omega \in \C$ by
\[
\omega \oplus_t \overline\omega(s) = \omega(s\wedge t) + \overline\omega(s)1_{[t,1]}(s).
\]
Recall the notation $h_n(t,q)=h(t,q/\sqrt{n})$. \\

\noindent\textbf{Proof of Theorem \ref{thm HL}}.
The case $t=1$ is trivial. Indeed, then $u_n(1,\omega) = Y_n(1,\sqrt{n}\omega) = F(\omega)$ for each $n$, which is seen to equal $u(1,\omega)=F(\omega)$. Assume henceforth that $t \in [0,1)$.
Note first that $(g^{(t)})_n = (g_n)^{(t)} =: g^{(t)}_n$. We let
\begin{align*}
u_n(t,\omega) &:= \rho^{g^{(t)}_n}\left(F\left(\omega \otimes_t \frac{W}{\sqrt{n}}\right)\right).
\end{align*}
Using Lemma \ref{le:BSDE-timedep}, we also have almost surely $Y_n(t,\omega) = u_n(t,\omega/\sqrt{n})$.
Since $F(\omega \otimes_t \cdot)$ is bounded and continuous, we may apply Theorem \ref{co:Schilder-finaltime} to get
\begin{align*}
\lim_{n\to\infty}u_n(t,\omega) &= \sup_{\overline\omega \in \C_0}\left(F(\omega \otimes_t \overline\omega) - \int_0^1g^{(t)}\left(s,\dot{\overline\omega}(s)\right)ds\right) \\
	&= \sup_{\overline\omega \in \C_0}\left(F(\omega \otimes_t \overline\omega) - (1-t)\int_0^1g\left(t+s(1-t),\frac{\dot{\overline\omega}(s)}{\sqrt{1-t}}\right)ds\right) \\
	&= \sup_{\overline\omega \in \C_0}\left(F(\omega \otimes_t \overline\omega) - \int_t^1g\left(s,\frac{1}{\sqrt{1-t}}\dot{\overline\omega}\left(\frac{s-t}{1-t}\right)\right)ds\right).
\end{align*}
Given $\overline\omega \in \C_0=\C_0[0,1]$, we may define
$\widetilde\omega \in \C_0[t,1]$ by $\widetilde\omega(s) \mapsto \sqrt{1-t}\,\overline\omega\left(\frac{s-t}{1-t}\right)$. Then $\omega \otimes_t \overline\omega = \omega \oplus_t \widetilde\omega$, and the map $\overline\omega \mapsto \widetilde\omega$ defines a bijection from $\C_0$ to $\C_0[t,1]$. Hence, the above reduces to $u(t,\omega)$.

To prove the final claim, let us first assume that $F$ is uniformly continuous. Using the fact that a convex risk measure is always $1$-Lipschitz with respect to the supremum norm (e.g.\ \cite[Lemma 4.3]{FS3dr}) we get
\begin{align*}
|Y_n(t,\omega) - Y_n(t,0)| &= |u_n(t,\omega/\sqrt{n}) - u_n(t,0)| \\
	&\le \left\|F\left(\frac{1}{\sqrt{n}}(\omega \otimes_t \cdot)\right) - F\left(\frac{1}{\sqrt{n}}(0 \otimes_t \cdot)\right)\right\|_\infty,
\end{align*}
which converges to zero by uniform continuity. This and the convergence for $u_n$ settles the uniformly continuous case.

Now, if $F$ is merely continuous, it is nevertheless the pointwise increasing limit of a sequence of bounded uniformly continuous (even Lipschitz) functions. Observing that both $Y_n(t)$ and $u(t,0)$ are increasing functions of $F$, we easily conclude from the uniformly continuous case that
\begin{align*}
\liminf_{n\to\infty}Y_n(t) \geq u(t,0), \ \ a.s.
\end{align*}
On the other hand, there is a uniformly bounded sequence $(F_m)$ of uniformly continuous functions decreasing to $F$. This time we can conclude that
\begin{align*}
\limsup_{n\to\infty}Y_n(t) \leq \inf_m \sup_{ \overline\omega \in \C_0[t,1] } \left( F_m( 0 \oplus_t \overline\omega ) - \int_t^1g(s, \dot{\overline\omega}(s))ds \right) \ \ a.s.
\end{align*}
It remains to bound the right-hand side from above by $u(t,0)$. For each $m \in \N$ find $\omega_m \in \C_0[t,1]$ such that
$$ \sup_{ \overline\omega \in \C_0[t,1] } \left( F_m( 0 \oplus_t \overline\omega ) - \int_t^1g(s, \dot{\overline\omega}(s))ds \right) \leq \frac{1}{m} + F_m( 0 \oplus_t \omega_m ) - \int_t^1g(s, \dot{\omega}_m(s))ds.$$ Since $(F_m)$ is uniformly bounded, we deduce (as we did for \eqref{pf:schilder-expbound} in the proof of Theorem \ref{co:Schilder-finaltime})
$$\sup_{m \in \N}\int_t^1g(s, \dot{\omega}_m(s))ds < \infty.$$
It is a consequence of Lemma \ref{le:H1-tightness} that there exists $\omega\in \C_0[t,1]$ absolutely continuous and such that for a subsequence (which we do not track) $\omega_m \to \omega$ uniformly, and $\liminf_m \int_t^1 g(s,\dot{\omega}_m(s))ds \geq \int_t^1 g(s,\dot{\omega}(s))ds $. 
On the other hand, since $F_m$ decreases pointwise to $F$, we have $F_m(0\oplus_t \omega_m)\to F(0\oplus_t \omega)$ by Dini's theorem.
We conclude that
$$\inf_m \sup_{ \overline\omega \in \C_0[t,1] } \left( F_m( 0 \oplus_t \overline\omega ) - \int_t^1g(s, \dot{\omega}(s))ds \right) \leq  F(0\oplus_t\omega) - \int_t^1 g(s,\dot{\omega}(s))ds \le u(t,0),$$
which completes the proof.\hfill\qedsymbol \\

\noindent\textbf{Proof of Theorem \ref{thm: Sanov time}}
The case $t=1$ is trivial: Because $(W_{(n,k)})_{k=1}^n$ are independent Wiener processes under $\W$, we conclude from the law of large numbers that $Y_n(1) = F\left(\frac{1}{n}\sum_{k=1}^n\delta_{W_{(n,k)}}\right)$ converges a.s.\ to $F(\W)$.

Henceforth, assume $t < 1$, so that $\lfloor nt \rfloor < n$ for all $n \in \N$.
First notice that
\begin{align*}
G_n^{(t)}(s,q) &:= (G_n)^{(t)}(s,q) = (1-t)G_n\left(t + s(1-t),\frac{q}{\sqrt{1-t}}\right) \\
	&= (1-t)g\left(nt + ns(1-t) - \lfloor nt + ns(1-t)\rfloor, \frac{q}{\sqrt{n(1-t)}}\right).
\end{align*}
Plugging in $t_n:=\lfloor nt\rfloor/n$, we find
\begin{align}
G_n^{(t_n)}(s,q) &= (1-t_n)g\left(\lfloor nt\rfloor  + s(n-\lfloor nt\rfloor) - \lfloor \lfloor nt\rfloor + s(n-\lfloor nt\rfloor)\rfloor, \frac{q}{\sqrt{n-\lfloor nt \rfloor}}\right) \nonumber \\
	&= (1-t_n)g\left(s(n-\lfloor nt\rfloor) - \lfloor s(n-\lfloor nt\rfloor)\rfloor, \frac{q}{\sqrt{n-\lfloor nt \rfloor}}\right) \nonumber \\
	&= (1-t_n)G_{n-\lfloor nt\rfloor}(s,q), \label{pf:timedep-Sanov-1}
\end{align}
where the second line used the identity $\lfloor k + c\rfloor = k +\lfloor c\rfloor$, valid for any integer $k$ and any $c \in \R$.
Define $L_n : \C \to \P(C)$ by
\[
L_n(\omega) := \frac{1}{n}\sum_{k=1}^n\delta_{\omega_{(n,k)}}.
\]
Using Lemma \ref{le:BSDE-timedep}, we write
\begin{align*}
Y_n(t_n,\omega) = \rho^{G^{(t_n)}_n}\left(F \circ L_n\left(\omega \otimes_{t_n} W\right)\right).
\end{align*}
Note that $(\omega \otimes_{t_n} W)_{(n,k)} \equiv \omega_{(n,k)}$ if $k \le nt_n = \lfloor nt \rfloor$, while for $k \ge \lfloor nt \rfloor + 1$ and $s \in [0,1]$ we have
\begin{align*}
(\omega \otimes_{t_n} W)_{(n,k)}(s) &= \sqrt{n(1-t_n)}\left(W\left(\frac{\frac{k-1+s}{n} - t_n}{1-t_n}\right) - W\left(\frac{\frac{k-1}{n} - t}{1-t_n}\right)\right) \\
	&= \sqrt{n - \lfloor nt \rfloor}\left(W\left(\frac{k-1+s - \lfloor nt \rfloor}{n - \lfloor nt \rfloor}\right) - W\left(\frac{k-1 - \lfloor nt \rfloor}{n - \lfloor nt \rfloor}\right)\right) \\
	&= W_{(n - \lfloor nt \rfloor, k - \lfloor nt \rfloor)}(s).
\end{align*}
Hence,
\begin{align}
L_n(\omega \otimes_{t_n} W) &= \frac{1}{n}\sum_{k=1}^{\lfloor nt \rfloor} \delta_{\omega_{(n,k)}} + \frac{1}{n}\sum_{k= \lfloor nt \rfloor + 1}^n\delta_{W_{(n - \lfloor nt \rfloor, k - \lfloor nt \rfloor)}} \nonumber \\
	&= t_n\frac{1}{\lfloor nt \rfloor}\sum_{k=1}^{\lfloor nt\rfloor}\delta_{\omega_{(n,k)}} + (1-t_n)L_{n - \lfloor nt \rfloor}(W). \label{pf:BSDE-Sanov-identity}
\end{align}
Assume first that $F$ is uniformly continuous. Under $\W$, $\omega_{(n,k)}$ for $k=1,\ldots,n$ are independent Brownian motions, and so as $n\to\infty$ the first term converges $\W$-a.s.\ by the law of large numbers to $t\W$.
Hence, it holds for $\W$-a.e.\ $\omega$ that {the existence of the limit
\begin{align*}
\lim_{n\to\infty}Y_n(t_n,\omega) &= \lim_{n\to\infty}\rho^{G^{(t_n)}_n}\left(F \circ L_n\left(\omega \otimes_{t_n} W\right)\right) ,
\end{align*}
is equivalent to the existence of the limit
	\begin{align*} 
	\lim_{n\to\infty}\rho^{G^{(t_n)}_n}\left(F\left(t\W + (1-t)L_{n - \lfloor nt \rfloor}(W)\right)\right),
\end{align*}
and if any of these exist, then they are equal. Indeed, from} the $1$-Lipschitz continuity of convex risk measures  \cite[Lemma 4.3]{FS3dr}, we have
\begin{align*}
&\left|\rho^{G^{(t_n)}_n}\left(F \circ L_n\left(\omega \otimes_{t_n} W\right)\right) - \rho^{G^{(t_n)}_n}\left(F\left(t\W + (1-t)L_{n - \lfloor nt \rfloor}(W)\right)\right)\right| \\
	&\quad \le \left\|F \circ L_n\left(\omega \otimes_{t_n} \cdot\right) - F\left(t\W + (1-t)L_{n - \lfloor nt \rfloor}(\cdot)\right)\right\|_\infty,
\end{align*}
with the right-hand side converging to zero thanks to the uniform continuity and boundedness of $F$, the law of large numbers, and the identity \eqref{pf:BSDE-Sanov-identity}.
Using this, equation \eqref{pf:timedep-Sanov-1}, and Theorem \ref{th:mainlimit} we compute the limit,
\begin{align*}
\lim_{n\to\infty}Y_n(t_n,\omega) &= \lim_{n\to\infty}\rho^{G^{(t_n)}_n}\left(F\left(t\W + (1-t)L_{n - \lfloor nt \rfloor}(W)\right)\right) \\
	&= \lim_{n\to\infty}\rho^{(1-t_n)G_{n - \lfloor nt \rfloor}}\left(F\left(t\W + (1-t)L_{n - \lfloor nt \rfloor}(W)\right)\right) \\
	&= \lim_{n\to\infty}\rho^{(1-t)G_{n - \lfloor nt \rfloor}}\left(F\left(t\W + (1-t)L_{n - \lfloor nt \rfloor}(W)\right)\right) \\
	&= \sup_{q \in \L_b}\left(F(t\W + (1-t)Q^q) - (1-t)\E\left[\int_0^1g(s,q(s))ds\right]\right).
\end{align*}
The  third equality, in which $(1-t_n)$ is replaced by $(1-t)$ in the superscript, follows from the estimate
\begin{align}
|\rho^{ag}(f)-\rho^{bg}(f)| \le \left(\frac{3\|f\|_\infty}{a} + \sup g^- + g(0) \right)|b-a|, \label{pf:BSDE-Sanov-g-estimate}
\end{align}
valid for any $g$ satisfying (TI), any bounded measurable $f$, and any $a,b \in (0,1]$, which we justify in the next paragraph. (Here $\sup g^- := \sup_{(t,q)}\max\{0,-g(t,q)\}$.)

To prove \eqref{pf:BSDE-Sanov-g-estimate} note that by monotonicity of $\rho^g$, it holds $\rho^g(f)\le \rho^g(\|f\|_\infty)= \|f\|_\infty+ \rho^g(0)\le \|f\|_\infty + \sup g^-$ and $\rho^g(f) \ge \E [f] - g(0)\ge -\|f\|_\infty - g(0)$.
Take note also of the easy identity $\rho^{cg}(f)=c\rho^g(f/c)$, valid for $c >0$.
Thus,
\begin{align*}
 	|\rho^{ag}(f) - \rho^{bg}(f)| &\le \left | a\rho^{g}\left(\frac fa\right ) - b\rho^g\left(\frac fa \right) \right | +  \left| b\rho^g\left(\frac fa \right) - b\rho^g\left(\frac fb \right) \right|\\
 	&\le \left|\rho^g\left(\frac fa \right)\right||a - b| + b\left\|\frac f a- \frac fb\right\|_{\infty}\\
 	&\le \left(3\frac{\|f \|_\infty}{a} + \sup g^- + g(0)  \right)|a-b|.
 \end{align*} 

We have now completed the proof under the extra assumption that $F$ is uniformly continuous. To conclude, we may drop this extra assumption by essentially the same monotone approximation arguments as in the proof of Theorem \ref{thm HL}, by relying again on Lemma \ref{le:H1-tightness}.

\section{On the PDE connection}
\label{sec BSDE}

The goal of this section is to briefly elaborate on the PDE results of Section \ref{sec:lim thm pde}. 
The basic lemma linking the functionals $\rho^g$ with PDEs is the following:

\begin{lemma} \label{le:minimalviscositysupersolution}
Let $f : \R^d \rightarrow \R$ be bounded and lower semicontinuous. Then the parabolic PDE
\begin{equation}
\label{eq:pde lemma}
		\begin{cases}
			\partial_tv(t,x) + \frac 12 \Delta v(t,x) + g(t, \nabla v(t,x)) = 0 \quad \text{on } [0,1]\times \R^d\\
			 v(1, x) = f(x), \quad \text{for } x\in \R^d
		\end{cases}
\end{equation}
admits  a minimal viscosity supersolution $v$. Moreover, $\rho^g(f(W(1)))=v(0,0)$.

{If $f\in C_b(\mathbb{R}^d)$ and $g^*(t,\cdot)$ is differentiable and there is a constant $C\ge 0$ such that 
\begin{equation*}
	|g^*(t,z)| \le C(1 + |z|^2) \quad \text{and}\quad |\partial_zg^*(t,z)|\le C(1 + |z|), \quad z \in \mathbb{R}^d,
\end{equation*}
then $v$ is the unique viscosity solution of \eqref{eq:pde lemma}.}
\end{lemma}
\begin{proof}
The existence of a minimal viscosity supersolution $v$ is shown in \cite[Theorem 5.2]{Drap-Main16}, where it is also shown that $v(0,0)=Y(0)$, where $(Y,Z)$ is the minimal supersolution of the BSDE \eqref{eq:bsde}. To complete the proof, simply recall from \eqref{eq rho = Y} that $Y(0)=\rho^g(f(W(1)))$.
{ When $g^*$ is of quadratic growth and $f\in C_b(\mathbb{R}^d)$, the existence of a unique viscosity solution $u$ follows by \cite[Theorems 3.2 and 3.8]{kobylanski01}.
By comparison, $v = u$.}
\end{proof}

Now, for each integer $n \ge 1$, consider the operator $\mathbb L_n$, taking bounded lower semicontinuous functions on $(\R^d)^n$ to bounded lower semicontinuous functions on $(\R^d)^{n-1}$, as follows. Given $F : (\R^d)^n \to \R$ and $(x_1,\ldots,x_{n-1}) \in (\R^d)^{n-1}$, we define $\mathbb L_nF(x_1,\dots,x_{n-1}) := v(0,0)$, where $v=v(t,x)$ is the minimal viscosity supersolution of the PDE
\begin{equation}
	\label{eq:pde n fold bla}
		\begin{cases}
			\partial_t v(t,x) + \frac 12 \Delta v(t,x) + g(t,\nabla v(t,x)) = 0 \quad \text{on } [0,1]\times \R^{d}\\
			 v(1, x) = F(x_1,\ldots,x_{n-1},x), \quad \text{for } x\in \R^d.
		\end{cases}
\end{equation}
By Lemma \ref{le:minimalviscositysupersolution}, the minimal viscosity supersolution $v$ exists, and we have $$\rho^g(F(x_1,\ldots,x_{n-1},W(1))) = v(0,0).$$
By definition,
\begin{align*}
\rho^g(F(x_1,\ldots,x_{n-1},W(1))) = \sup_{Q \in \Q}\E^Q\left[F\left(x_1,\ldots,x_{n-1},W(1) + \int_0^1q^Q(t)dt\right) - \int_0^1g(t,q^Q(t))dt\right].
\end{align*}
Because $F$ is lower semicontinuous and bounded, this exhibits $\rho^g(F(x_1,\ldots,x_{n-1},W(1)))$ as the supremum of lower semicontinuous functions of $(x_1,\ldots,x_{n-1})$. Hence, $\mathbb L_n$ is well defined and indeed maps bounded lower semicontinuous functions of $(\R^d)^n$ to bounded lower semicontinuous functions of $(\R^d)^{n-1}$.
For $n=1$, we interpret $\mathbb L_1$ as mapping from bounded lower semicontinuous functions of $\R^d$ to real numbers.
The composition $\mathbb L_1\cdots\mathbb L_{n-1}\mathbb L_n$ then maps a function on $(\R^d)^n$ to a real number. 

\begin{proposition} \label{pr:PDEiteration}
For a function $F \in C_b(\P(\R^d))$, define $F^n : (\R^d)^n \rightarrow \R$ by
\[
F^n(x_1,\dots,x_n):= nF\left( \frac{1}{n}\sum_{i=1}^n\delta_{x_i} \right).
\]
Then, defining $Q^q_1 := \W \circ (W(1) + \int_0^1q(t)dt)^{-1}$ for $q \in \L_b$ as the time-$1$ marginal of $Q^q$,
\begin{align}
\lim_{n\rightarrow\infty}\frac{1}{n}\mathbb L_1\cdots\mathbb L_{n-1}\mathbb L_nF^n = \sup_{q \in \L_b}\left(F(Q^q_1) - \E\left[\int_0^1g(t,q(t))dt\right]\right). \label{eq:limit pde mfg}
\end{align}
\end{proposition}
\begin{proof}
Recall the definition of $\rho^g_n$ from Section \ref{sec: stoch rep rho n}.
For a bounded lower semicontinuous function $f$ on $(\R^d)^n$, define $\widetilde{f} \in B_b(\C^n)$ by setting $\widetilde{f}(\omega_1,\ldots,\omega_n) = f(\omega_1(1),\ldots,\omega_n(1))$, and note that we have
\[
\mathbb L_1\cdots\mathbb L_{n-1}\mathbb L_nf = \rho^g_n(\widetilde{f}) = n\rho^{G_n}\left(\frac{1}{n}\widetilde{f}(W_{(n,1)},\ldots,W_{(n,n)})\right).
\]
Indeed, the first equality is just the definition of $\rho^g_n$, while the second is Proposition \ref{pr:rhon}.
In particular, we may write
\[
\widetilde{F}^n(\omega_1,\ldots,\omega_n) = F^n(\omega_1(1),\ldots,\omega_n(1)) = n\,F\left( \frac{1}{n}\sum_{i=1}^n\delta_{\omega_i(1)} \right),
\]
and thus
\begin{align*}
\frac{1}{n}\mathbb L_1\cdots\mathbb L_{n-1}\mathbb L_nF^n &= \rho^{G_n}\left(\widetilde{F}^n(W_{(n,1)},\ldots,W_{(n,n)})\right) = \rho^{G_n}\left(F\left(\frac{1}{n}\sum_{k=1}^n\delta_{W_{(n,k)}(1)}\right)\right).
\end{align*}
Conclude from Theorem \ref{th:mainlimit}.
\end{proof}

\begin{remark}
The right-hand side of \eqref{eq:limit pde mfg} can be further rewritten as $\sup_{\nu\in\mathcal P(\R^d)}\left\{ F(\nu)- I(\nu) \right\}$, where $I(\nu):=\inf \left\{\E^Q[\int_0^1g(t,q^Q(t))dt] :\, Q\in\mathcal P^*_1(\C),\, Q\circ \omega(1)^{-1}=\nu \right\}$ is a Schr\"odinger-type problem under the classical observable (as discussed in Section \ref{sec:schilder.to.schroe}).
\end{remark}
We finally turn our attention to the formalization of the heuristics given in Section \ref{sec PDE Schilder}. The novelty here lies in the ``stochastic'' proof, involving our BSDE limit theorems  which allow to bypass the regularity conditions often made on the coefficients of the PDE.
\begin{proposition}
	\label{prop:hopf-lax}
		Let $f \in C_b(\mathbb{R}^d)$.
		The PDE
		\begin{equation}\label{eq HJB un}
		\begin{cases}
			\partial_tu_n(t,x) + \frac{1}{2n} \Delta u_n(t,x) + g^*(t, \nabla u_n(t,x)) = 0 \quad \text{on } [0,1]\times \R^d\\
			 u_n(1, x) = f(x), \quad \text{for } x\in \R^d.
		\end{cases}
		\end{equation}
		admits a minimal viscosity supersolution $u_n$.
		Moreover, $u_n\to u$ pointwise, where $u$ is the function given by
		\begin{equation*}
			u(t,x) = \sup_{\omega \in \C_{0}[t,1]}\left( f(x + \omega(1)) - \int_t^1g(s, \dot\omega(s))\,ds \right).
		\end{equation*}
		When $g(t, q) = g(q)$ does not depend on $t$, then the function $u$ reduces to the Hopf-Lax-Oleinik formula \eqref{eq:hopf lax form} and if in addition $g$ is real-valued and $f$ Lipschitz continuous, then $u$ is the unique viscosity solution of \eqref{eq HJ u}.
\end{proposition}
\begin{proof}
		Let $(t, x) \in [0,1]\times \mathbb{R}^d$ be fixed and put $X^{t,x}_n(s):=x + \frac{1}{\sqrt{n}}(W(s)- W(t))$, $s\ge t$.
		By \cite{Drap-Main16} the function $u_n(t,x):= Y_n(t)$ is the minimal supersolution of the PDE \eqref{eq HJB un}, where $(Y_n, Z_n)$ is the minimal supersolution of the BSDE with generator $g^*_n$ and terminal condition $f(X^{t,x}_n(1))$.
		Let $F:\C\to \mathbb{R}$ be given by $F(\omega) = f(x + \omega(1) - \omega(t))$.
		By Theorem \ref{thm HL} and the fact that $Y_n(t)$ is deterministic, it holds
		\begin{align*}
			u_n(t,x) = Y_n(t)& \to  \sup_{ \omega \in \C_0[t,1] } \left( F( 0 \oplus_t \omega ) - \int_t^1g(s, \dot\omega(s))ds \right)\\
							 & = \sup_{ \omega \in \C_0[t,1] } \left( f( x + \omega(1) ) - \int_t^1g(s, \dot\omega(s))ds \right) = u(t,x).
		\end{align*}
Now, when $g$ is time-independent, the Hopf-Lax-Oleinik formula \eqref{eq:hopf lax form} follows from Jensen's inequality. Granting the additional assumptions on $g$ and $f$, it is classical that the Hopf-Lax-Oleinik formula is the unique viscosity solution of \eqref{eq HJ u}; see \cite[Theorem 10.3]{Evans1998}.
\end{proof}

\section{Some extensions of the limit theorems} \label{se:extensions}

In this section we describe two extensions of the main theorems. First, we show how to strengthen the topology used in Theorems \ref{th:mainlimit} and \ref{thm: Sanov time} to the $1$-Wasserstein topology, which allows us to derive a Cram\'er-type theorem. Second, we incorporate a random initial position for $W(0)$, which has thus far been assumed to be zero.

{
\subsection{Extension to stronger topologies}
\label{sec stronger topo} 
Recall that $\mathcal{P}(\C)$ denotes the set of Borel probability measures on $\C$. Define $\mathcal{W}_1$ to be the $1$-Wasserstein metric on the space
\[
\P_1(\C) := \left\{Q \in \P(\C) : \int_\C \|\omega\|_\infty\,Q(d\omega) < \infty\right\},
\]
where we recall $\|\cdot\|_\infty$ denotes the supremum norm on $\C$. That is, $\mathcal{W}_1(Q,Q')$ is the infimum over all $\overline Q \in \P(\C \times \C)$ with marginals $Q$ and $Q'$ of the quantity $\int \|\omega-\omega'\|_\infty \overline Q(d\omega,d\omega')$. Recall that $\tilde{\alpha}^g$ was defined in \eqref{def:tildealpha}, and as usual we tacitly assume $g$ satisfies (TI).

\begin{lemma} \label{le: Wr compact level sets}
The sub-level sets of $\tilde{\alpha}^g$ are $\mathcal{W}_1$-compact. More precisely, for every $a \in \R$ the set $\Lambda_a := \{Q \in \P(\C) : \tilde{\alpha}^g(Q) \le a\}$ is contained in $\P_1(\C)$ and is compact in the $\mathcal{W}_1$-topology.
\end{lemma}
\begin{proof}
Noting that $g$ is bounded from below and $\tilde{\alpha}^{g+c}=\tilde{\alpha}^g+c$ for constants $c \in \R$, we may assume without loss of generality that $g \ge 0$.
Fix $a \in \R$. We know from  Lemma \ref{thm technical} that $\Lambda_a$ is compact in the topology of weak convergence. It suffices to show (see \cite[Theorem 7.12]{villani2003topics}) that
\begin{align}
\lim_{r\rightarrow\infty}\sup_{Q \in \Lambda_a}\E^Q[\|W\|_\infty 1_{\{\|W\|_\infty  \ge 2r\}}] = 0. \label{pf:Wr0}
\end{align}
By Assumption (TI), for each $c > 0$ we may find $N > 0$ such that $g(t,q) \ge c|q|$ whenever $|q| \ge N$. Clearly $\Lambda_a \subset \P^*$. For $Q \in \Lambda_a$, by definition, $W^Q(t) := W(t) - \int_0^tq^Q(s)ds$ is a $Q$-Brownian motion.  Hence, for any $r > 1$,
\begin{align}
\E^Q[\|W\|_\infty 1_{\{\|W\|_\infty  \ge 2r\}}] &\le \E^Q[\|W\|_\infty 1_{\{\|W^Q\|_\infty  \ge r\}}] + \E^Q[\|W\|_\infty 1_{\{\int_0^1|q^Q(t)|dt  \ge r\}}]. \label{pf:Wr1}
\end{align}
For the first term, we make the estimate
\begin{align}
\E^Q[\|W\|_\infty 1_{\{\|W^Q\|_\infty  \ge r\}}] &\le \E^Q[\|W^Q\|_\infty 1_{\{\|W^Q\|_\infty  \ge r\}}] + \E^Q\left[\int_0^1|q^Q(t)|dt 1_{\{\|W^Q\|_\infty  \ge r\}}\right] \nonumber  \\
	&\le \E^\W[\|W\|_\infty1_{\{\|W\|_\infty \ge r\}}] + N\E^Q\left[1_{\{\|W^Q\|_\infty  \ge r\}}\right] \nonumber \\
		&\quad  + \frac{1}{c}\E^Q\left[\int_0^1 g(t,q^Q(t)) \,dt \,1_{\{\|W^Q\|_\infty  \ge r\}}\right]  \nonumber  \\
	&\le (1+N)\E^\W[\|W\|_\infty1_{\{\|W\|_\infty \ge r\}}] + \frac{1}{c}\E^Q\left[\int_0^1g(t,q^Q(t))dt\right]. \label{pf:Wr2}
\end{align}
We bound the second term of \eqref{pf:Wr1} similarly:
\begin{align}
\E^Q[\|W\|_\infty 1_{\{\int_0^1q^Q(t)dt  \ge r\}}] &\le \E^Q[\|W^Q\|_\infty 1_{\{\int_0^1|q^Q(t)|dt  \ge r\}}] + \E^Q\left[\int_0^1|q^Q(t)|dt 1_{\{\int_0^1|q^Q(t)|dt  \ge r\}}\right] \nonumber \\
	&\le \E^\W[\|W\|_\infty^2]^{1/2}Q\left(\int_0^1|q^Q(t)|dt  \ge r\right)^{1/2} \nonumber \\
	&\quad+ \frac{1}{c}\E^Q\left[\int_0^1g(t,q^Q(t))dt \right] + NQ\left(\int_0^1|q^Q(t)|dt  \ge r\right). \label{pf:Wr3}
\end{align}
Lastly, note that the definition of $\Lambda_a$ and Assumption (TI) ensure that 
\begin{align}
\lim_{r\rightarrow\infty}\sup_{Q \in \Lambda_a}Q\left(\int_0^1|q^Q(t)|dt  \ge r\right) = 0.
\end{align}
Combining this with \eqref{pf:Wr2}-\eqref{pf:Wr3} and returning to \eqref{pf:Wr1}, we deduce \eqref{pf:Wr0} since $c$ was arbitrary.
\end{proof}

\begin{corollary} \label{co:Sanov W1}
The conclusions of Theorems \ref{th:mainlimit} and \ref{thm: Sanov time} hold for any  $F \in C_b(\mathcal{P}_1(\C))$, where $\mathcal{P}_1(\C)$ is equipped with the metric $\mathcal{W}_1$, with the suprema over $Q \in \Q$ replaced by $Q \in \Q \cap \P_1(\C)$.
\end{corollary}
\begin{proof}
The proofs are exactly the same as those of Theorems \ref{th:mainlimit} and Theorem \ref{thm: Sanov time}, with only minor points to check. In light of Lemma \ref{le: Wr compact level sets}, we may apply the more general \cite[Theorem 3.1]{Lac-Sanov} in place of \cite[Theorem 1.1]{Lac-Sanov} to conclude that
\begin{align*}
\lim_{n\rightarrow\infty}\frac{1}{n}\tilde{\rho}^g_n(nF \circ L_n) = \sup_{Q \in \P_1(\C)}(F(Q) - \tilde{\alpha}^g(Q)).
\end{align*}
The only point worth checking is that
\[
\sup_{Q \in \P_1(\C)}(F(Q) - \tilde{\alpha}^g(Q)) = \sup_{Q \in \Q \cap \P_1(\C)}(F(Q) - \alpha^g(Q))
\]
holds when $F$ is merely $\mathcal{W}_1$-continuous, but the same argument as in the proof of Theorem \ref{th:mainlimit} works: If $\tilde{\alpha}^g(Q) < \infty$, then there exists $Q_n \in \Q$ such that $Q_n \to Q$ weakly and $\limsup_n\alpha^g(Q_n) \le \tilde{\alpha}^g(Q)$. Deduce from Lemma \ref{le: Wr compact level sets} that $\{Q_n\}$ is $\mathcal{W}_1$-precompact and thus $\mathcal{W}_1(Q_n,Q) \to 0$. Hence, $F(Q_n)\to F(Q)$, and the above identity follows.
\end{proof}

{
As a consequence of Corollary \ref{co:Sanov W1}, we provide the following Cr\'amer-type limit theorem:

\begin{corollary} \label{co:Schilder-finaltime weird}
For every $F \in C_b(\C)$, we have 
\begin{align} \label{eq weird Schilder}
\lim_{n\rightarrow\infty}\,\rho^{G_n}\left(F \left(\frac 1n \sum_{k=1}^n{W_{(n,k)}}\right)\right) 
= \sup_{\omega \in \C_0} \left (F(\omega) - \int_0^1 g(t,\dot{\omega}(t))dt\right).
\end{align}
\end{corollary}

\begin{proof}
Apply Corollary \ref{co:Sanov W1} to the $\mathcal{W}_1$-continuous function $\P_1(\C) \ni Q \mapsto F\left(\int_\C \omega \,Q(d\omega)\right)$, where the integral is understood in the Bochner sense, to get
\begin{align*}
\lim_{n\rightarrow\infty}\,\rho^{G_n}\left(F \left(\frac 1n \sum_{k=1}^n{W_{(n,k)}}\right)\right) 
= \sup_{Q \in \P_1(\C)}\left(F\left(\int_{\C} \omega\,Q(d\omega)\right) - {\alpha}^g(Q)\right).
\end{align*}
By the arguments in the proof of Corollary \ref{co:Sanov W1}, the above expression is equal to
\[
\sup_{Q \in \mathcal P^*}\left(F\left(\int_{\C} \bar{\omega}\,Q(d\bar{\omega})\right) - \tilde{\alpha}^g(Q)\right) = \sup_{\omega \in \C_0}\left(F(\omega) - I(\omega)\right),
\]
where we define
\begin{align*}
I(\omega) &:= \inf\left\{\tilde\alpha^g(Q): Q \in \P^* \cap \P_1(\C), \ \int_{\C} \bar{\omega}\,Q(d\bar{\omega}) = \omega\right\}.
\end{align*}
Indeed, we may restrict the supremum to $\P^* \cap \P_1(\C)$ as opposed to $\P^*$ because $\tilde{\alpha}^g(Q)=\infty$ for $Q \notin \P_1(\C)$ by Lemma \ref{le: Wr compact level sets}.
We need only show that
\[
I(\omega) = 
 \begin{cases}
 \int_0^1 g(t,\dot{\omega}(t))dt & \text{if } \omega \text{ is absolutely continuous}\\
 \infty &\text{otherwise}.
 \end{cases}
\]
Noting that $\E^Q[W(t)] = \int_0^t\E^Q[q^Q(s)]ds$ for $Q \in \P^* \cap \P_1(\C)$, we have 
\begin{align*}\textstyle
 I(\omega) &=\textstyle \inf\left\{\E^Q\left[\int_0^1g(t,q^Q(t))dt\right] : Q \in \P^* \cap \P_1(\C), \ \int_0^t\E^Q[q^Q(s)]ds = \omega(t), \ \forall t \in [0,1]\right\}.
\end{align*}
Now, fix $\omega \in \C$.
Jensen's inequality yields 
\begin{align}
 \E^Q\left[\int_0^1g(t,q^Q(t))dt\right] \geq \int_0^1g(t,\E^Q[q^Q(t)])dt = \int_0^1g(t,\dot{\omega}(t))dt, \label{pf:Schilder1}
\end{align}
for any $Q \in \P^*$ for which $\int_0^t\E^Q[q^Q(s)]ds = \omega(t)$ for all $t \in [0,1]$. If $\omega$ is absolutely continuous, then we can define $Q= \W\circ(W+\int_0^\cdot \dot{\omega}(t)dt)^{-1}$ so that $Q \in \P^* \cap \P_1(\C)$ with $q^Q(t) = \dot{\omega}(t)$ for all $t \in [0,1]$. We conclude that, for $\omega$ absolutely continuous,
\[
I(\omega) = \int_0^1g(t,\dot{\omega}(t))dt.
\]
On the other hand, if $\omega$ is not absolutely continuous, then there cannot exist $Q \in \P^* \cap \P_1(\C)$ with $\int_0^t\E^Q[q^Q(s)]ds = \omega(t)$ for all $t \in [0,1]$.
\end{proof}
 
\begin{remark}
Comparing \eqref{eq Schilder familiar} and \eqref{eq weird Schilder}, we find that
\begin{align*}
\lim_{n\to\infty}\rho^{g_n}\left(F\left(\frac{W}{\sqrt{n}}\right)\right) = \lim_{n\rightarrow\infty}\,\rho^{G_n}\left(F \left(\frac 1n \sum_{k=1}^n{W_{(n,k)}}\right)\right) .
\end{align*}
If $F(\omega)=f(\omega(1))$ depends only on the final value, then these quantities are even equal for each $n$, without taking a limit (by telescoping sum).
This may at first seem unsurprising (at least for time-independent $g$) because $W/\sqrt{n}$ and $\frac 1n \sum_{k=1}^n{W_{(n,k)}}$ have the same law for each $n$.
In general, however, we do not expect pre-limit equality except when $\rho^g$ is law-invariant.
By \cite{kupper02}, the functional $\rho^g$ is law-invariant essentially only when $g(t,q)=c|q|^2$ for $c \in (0,\infty]$, with the convention $0\cdot\infty:=0$.
\end{remark}

}

{
\subsection{Extensions to non-trivial initial positions}
\label{sec extension of results}
Preparing for our study of Schr\"odinger problems, we now extend some of our results to allow the Brownian motion to have a (constant) volatility different than $1$ as well as a random, non-zero initial position.

We fix throughout this section the function $g$ satisfying assumption (TI), and we will omit it from our soon-to-be cluttered superscripts.
Recall that $\W$ denotes Wiener measure on $\C$ and $W$ denotes the canonical process (identity map) on $\C$. 
For $Q\in \mathcal{P}(\C)$ we take a regular kernel $(Q^{\omega(0)=x})_{x \in \R^d}$ so by disintegration
\[
Q(\cdot) = \int_{\R^d} Q^{\omega(0)=x}(\cdot)Q^0(dx),
\]
where $Q^{\omega(0)=x} \in \P(\C)$ is supported on the set $\C_x := \{\omega \in \C : \omega(0)=x\}$ and $Q^0$ is the time-zero marginal of $Q$.

We are given $\mu \in \P(\R^d)$. Recalling that $\W_\epsilon = \W \circ (\sqrt{\epsilon}W)^{-1}$, we define
\[
{\W^{\omega(0)\sim\mu}_\epsilon(\cdot)}: = \int_{x\in \R^d}\W^{\omega(0)=x}_\epsilon(\cdot)\mu(dx),
\]
namely the law of a Brownian motion with starting distribution $\mu$ and instantaneous variance (i.e.\ volatility) equal to $\epsilon$. 

For $Q \in \P(\C)$ with $Q \ll \W^{\omega(0)\sim\mu}_\epsilon$ and $Q^0=\mu$, we define $q^Q_\epsilon$ as the unique progressively measurable process satisfying
\begin{align*}
\frac{dQ}{d\W^{\omega(0)\sim\mu}_\epsilon}= \exp\left ( \frac{1}{\epsilon}\int_0^1 q^Q_\epsilon(t)dW(t) - \frac{1}{2\epsilon}\int_0^1 |q^Q_\epsilon(t)|^2dt  \right ).
\end{align*}
Then, for $Q \in \P(\C)$ we define 
\begin{align*}
\alpha^\mu_\epsilon(Q) := \begin{cases}
\E^Q\left[\int_0^1g(t, q^Q_\epsilon(t))dt\right] &\text{if } Q \ll \W^{\omega(0)\sim\mu}_\epsilon, \ Q^0 = \mu \\
+\infty &\text{otherwise}.
\end{cases}
\end{align*}
It is straightforward to check that
\begin{align}
\alpha^\mu_\epsilon(Q) = \begin{cases}
\int_{\R^d}\alpha^{\delta_x}_\epsilon(Q^{\omega(0)=x})\mu(dx) &\text{if } Q \ll \W^{\omega(0)\sim\mu}_\epsilon, \ Q^0 = \mu \\
+\infty &\text{otherwise}.
\end{cases} \label{def:alphamu-pointwise}
\end{align}

On the dual side, for $F \in B_b(\C)$, we define
\begin{align*}
\rho^\mu_\epsilon(F) &:= \sup_{Q \in \P(\C)}\left(\E^Q[F] - \alpha^\mu_\epsilon(Q)\right) \\
	&= \sup_{Q \ll \W^{\omega(0)\sim\mu}_\epsilon, \ Q^0=\mu}\E^Q\left[F(W) - \int_0^1g(t, q^{Q}_\epsilon(t))dt\right].
\end{align*}
Let us recall the notation for $\mathcal{P}^*_\epsilon(\C)$ in  Section \ref{sec:schilder.to.schroe}, as well as $Q\mapsto q^Q$ defined there. Define
$$\tilde{\alpha}^\mu_\epsilon(Q):= \left \{ 
\begin{array}{ll}
\E^Q\left [\int_0^1 g(t,q^Q(t))dt \right ] & \text{ if $Q\in \mathcal{P}^*_\epsilon(\C)$ and $Q^0=\mu$}\\
+\infty & \text{ otherwise},
\end{array}
\right .
$$
 and introduce analogously
$$\tilde{\rho}^\mu_\epsilon(F) := \sup_{Q \in \P(\C)}\left(\E^Q[F] - \tilde{\alpha}^\mu_\epsilon(Q)\right).$$

We ask the reader to bear in mind that, whenever we use $g$ or any other function as superscript for $\alpha$ or $\rho$ (resp. $\tilde{\alpha}$ or $\tilde{\rho}$), we mean it in the sense of Section \ref{sec setting} (resp. Section \ref{se:compactness-alpha}), with the starting distribution being fixed to $\delta_0$. On the other hand, whenever we use $\mu$ or any other measure as supercript for $\alpha,\rho,\tilde{\alpha},\tilde{\rho}$, we mean it in the sense presented in the current section (the function $g$ being fixed). 

Let us first present the analogue to Lemmas \ref{thm technical} and \ref{lem concidence} (which took care of $\mu=\delta_0$ and $\epsilon=1$) in the present setup:

\begin{lemma}\label{thm technical gen}
The functional $\tilde{\alpha}_\epsilon^\mu$ is convex and lower semicontinuous (with respect to weak convergence), and its sub-level sets are weakly compact in this topology. Furthermore, for $Q\in\mathcal{P}^*_\epsilon(\C)$ with $Q^0=\mu$ we have 
$$\tilde{\alpha}_\epsilon^\mu(Q)=\sup_{F\in B_b(\C)}\left(\E^Q[F]-\tilde\rho_\epsilon^\mu(F)\right)={\sup_{F\in C_b(\C)}\left(\E^Q[F]-\tilde\rho_\epsilon^\mu(F)\right)},$$
and, on the other hand, for $F:\C\to\R$ bounded lower-semicontinuous we have $$ \rho_\epsilon^\mu(F)=\tilde\rho_\epsilon^\mu(F). $$
\end{lemma}

We omit the proof, since it boils down to the same arguments as for Lemmas \ref{thm technical} and \ref{lem concidence}. The key point of this section is the following proposition, for which we recall the notation $\C_x = \{\omega \in \C : \omega(0)=x\}$:

\begin{proposition} \label{th:Schilder-initial}
Let $\mu \in \P(\R^d)$ and $\epsilon > 0$. For $F: \C\to \R$ measurable and bounded we have
\begin{align}\label{eqdisintegrationrho}
\rho^\mu_\epsilon(F) = \int_{\R^d} \rho^{\delta_x}_\epsilon(F)\,\mu(dx).
\end{align}
For $F \in C_b(\C)$ we further have
\begin{align}\label{eqlimitgeneralvar}
\lim_{\epsilon \downarrow 0}\,\rho^\mu_\epsilon\left(F\right) = \int_{\R^d}\,\, \sup_{\omega \in \C_x}\left(F(\omega) - \int_0^1g(t,\dot{\omega}(t))dt\right)\mu(dx).
\end{align}
\end{proposition}
\begin{proof}
Using \eqref{def:alphamu-pointwise}, we have
\begin{align}
\rho^\mu_\epsilon(F) &= \sup_{Q \in \P(\C)}\left(\E^Q[F] - \alpha^\mu_\epsilon(Q)\right) \nonumber \\
	&= \sup_{Q \ll \W_\epsilon^{\omega(0)\sim\mu}, \ Q^0=\mu}\left(\E^Q[F] - \int_{\R^d}\alpha^{\delta_x}_\epsilon(Q^{\omega(0)=x})\mu(dx)\right) \nonumber \\
	&= \sup_{Q \ll \W_\epsilon^{\omega(0)\sim\mu}, \ Q^0=\mu}\int_{\R^d}\left(\E^{Q^{\omega(0)=x}}[F] - \alpha^{\delta_x}_\epsilon(Q^{\omega(0)=x})\right)\mu(dx) \label{pf:Schilder-initial1} \\
	&\le \int_{\R^d}\rho^{\delta_x}_\epsilon(F)\,\mu(dx). \nonumber 
\end{align}
To prove the reverse relies on a careful application of a standard measurable selection argument.
A straightforward transformation of \eqref{eq BBD} yields
\begin{align}\label{eq BBD general var}
\rho^{\delta_x}_\epsilon(F) = \sup_{q\in \L_b}\E^{\W^{\omega(0)=0}}\left [F\left ( x+\sqrt{\epsilon}\, W+\int_0^\cdot q(t)dt  \right ) - \int_0^1 g(t,q(t))dt  \right ].
\end{align}
Note that $\L_b$ is a Borel subset of the (separable metric) space $\L^2$ of square-integrable progressively measurable processes, and that the map 
\[
\R^d\times \L_2 \ni(x,q)\mapsto \E^{\W^{\omega(0)=0}}\left [F\left ( x+\sqrt{\epsilon}\, W+\int_0^\cdot q(t)dt  \right ) - \int_0^1 g(t,q(t))dt  \right ]
\]
is measurable. We may apply standard analytic set theory \cite[Proposition 7.47]{BertsekasShreve} to conclude that $x \mapsto \rho^{\delta_x}_\epsilon(F)$ is upper semianalytic and, in particular, universally measurable.
The integral in the right-hand side of \eqref{eqdisintegrationrho} is thus well defined, since further $\rho^{\delta_x}_\epsilon(F)$ is bounded by the bounds of $F$ and $g$. 
By \cite[Proposition 7.50]{BertsekasShreve}, there exists a universally measurable $\eta$-approximate optimizer $q^x \in \L_b$ in \eqref{eq BBD general var}, for any $\eta > 0$. Letting $Q_x= \W \circ \left(x+\sqrt{\epsilon}\, W+\int_0^\cdot q^x(t)dt  \right)^{-1}$, we check that the probability measure $Q = \int_{x\in \R^d}Q_x\,\mu(dx)$
satisfies $Q \ll \W_\epsilon^{\omega(0)\sim\mu}$, $Q^0=\mu$, and $Q^{\omega(0)=x} = Q_x$. Moreover, by design,
\begin{align*}
\rho^{\delta_x}_\epsilon(F) - \eta &\le \E^{Q_x}[F] - \alpha^{\delta_x}_\epsilon(Q_x)  = \E^{Q^{\omega(0)=x}}[F] - \alpha^{\delta_x}_\epsilon(Q^{\omega(0)=x}).
\end{align*}
Hence, using the expression \eqref{pf:Schilder-initial1} for $\rho_\epsilon^\mu(F)$, we deduce $\int_{\R^d} \rho^{\delta_x}_\epsilon(F)\,\mu(dx) - \eta \le \rho_\epsilon^\mu(F)$. As $\eta > 0$ was arbitrary, this proves \eqref{eqdisintegrationrho}.

Now we show \eqref{eqlimitgeneralvar}. The key is to observe from \eqref{eq BBD general var} that
$$\rho^{\delta_x}_\epsilon(F) = \rho^{g_{{\epsilon}}}\left( F_\epsilon^x\right),$$
where $g_\epsilon(q):=g(\sqrt{\epsilon}q)$ and $F_\epsilon^x(\omega):=F(x+\sqrt{\epsilon}\omega)$. Indeed,
\begin{align*}
\rho^{g_{{\epsilon}}}\left( F_\epsilon^x\right) &=\sup_{q \in \L_b} \E^{\W_1^{\omega(0)=0}}\left[ F\left(x+\sqrt{\epsilon} W + \sqrt{\epsilon} \int_0^\cdot q(s) ds \right) - \int_0^1 g(t,\sqrt{\epsilon}  q(t))dt \right ]\\
&= \sup_{q \in \L_b} \E^{\W_\epsilon^{\omega(0)=0}}\left[ F\left(x+ W + \int_0^\cdot  q(s) ds \right) - \int_0^1 g(t,q(t))dt \right ]\\
&= \sup_{Q\ll \W_\epsilon^{\omega(0)=0}} \E^Q\left[ F\left(x+ W \right) - \int_0^1 g(t,q^Q(t))dt \right ]\\
&= \sup_{Q\ll \W_\epsilon^{\omega(0)=x}} \E^Q\left[ F\left( W \right) - \int_0^1 g(t,q^Q(t))dt \right ]\\
&= \rho^{\delta_x}_\epsilon(F),
\end{align*}
where we used \eqref{eq BBD} in the first and third equalities. Thus Theorem \ref{co:Schilder-finaltime} implies
$$\lim_{\epsilon \downarrow 0} \rho^{\delta_x}_\epsilon(F) = \sup_{\omega\in \C_0}\left ( F(x+\omega) - \int_0^1 g(t,\dot{\omega}(t))dt  \right )= \sup_{\omega \in \C_x}\left(F(\omega) - \int_0^1g(t,\dot{\omega}(t))dt\right).$$
With this at hand, we conclude by \eqref{eqdisintegrationrho} and dominated convergence.
\end{proof}

\section{Application to Schr\"odinger-type problems}
\label{sec Schroedinger details}

Our aim is to prove the results stated in Section \ref{sec:schilder.to.schroe}. We first need some preparatory lemmas. We carry on with the notation of Section \ref{sec extension of results}, recalling the convention that $\int_0^1g(t,\dot\omega(t))dt=\infty$ if $\omega$ is not absolutely continuous. We introduce the following very important functional
\[
\alpha^\mu_0(Q):= \E^Q\left [ \int_0^1 g(t,\dot W(t) )dt \right ].
\]
We also recall that $Z$ is a separable Banach space (of observations) and that $H: \C\to Z$, the observable, is a continuous linear operator. 

The following $\Gamma$-convergence type result is a crucial technical step, and part \emph{(i)} of it relies on our Schilder-type result (Proposition \ref{th:Schilder-initial}) in an essential way. Recall that $\W_\epsilon = \W \circ (\sqrt{\epsilon}W)^{-1}$ denotes the law of a standard Brownian motion times $\sqrt{\epsilon}$.

\begin{lemma}\label{lem convolution}
As $\epsilon \downarrow 0$, $\tilde{\alpha}^\mu_\epsilon$ converges to the function $\alpha^\mu_0$ in the sense of $\Gamma$-convergence. This means that for all $Q \in \mathcal{P}(\C)$:
\begin{enumerate}
\item[(i)] Whenever $Q_\epsilon \to Q$, then $$\liminf\limits_{\epsilon\downarrow 0} \tilde{\alpha}^\mu_\epsilon(Q_\epsilon) \geq \alpha^\mu_0(Q).$$
\item[(ii)] There exists some $\tilde{Q}_\epsilon\to Q$ such that
$$\limsup\limits_{\epsilon\downarrow 0} \tilde{\alpha}^\mu_\epsilon(\tilde{Q}_\epsilon) \leq \alpha^\mu_0(Q).$$
\end{enumerate}
Moreover, the sequence $\{\tilde{Q}_\epsilon\}$ in $(ii)$ can be explicitly taken as $\tilde{Q}_\epsilon := Q * \W_\epsilon$.
\end{lemma}

\begin{proof}
We first show $(ii)$. We may assume $Q$ is such that $\alpha^\mu_0(Q)<\infty$, and take
$$\tilde{Q}_\epsilon:= Q * \W_\epsilon:= \int_{\C} \W\circ(\bar{\omega} +\sqrt{\epsilon}W)^{-1}Q(d\bar{\omega}).$$
To be completely clear, this means 
$$\int_\C Fd\tilde{Q}_\epsilon = \int_\C\int_\C F(\bar{\omega} +\sqrt{\epsilon} \,\omega )\W(d\omega)Q(d\bar{\omega}).$$
It is readily verified, via Lebesgue dominated convergence, that $\tilde{Q}_\epsilon \to Q$ weakly. Since $\alpha^\mu_0(Q)<\infty$ it follows that $Q$ is concentrated on absolutely continuous paths, so as a consequence $\tilde{Q}_\epsilon\in \mathcal{P}^*_\epsilon(\C)$. Furthermore, $\tilde{Q}_\epsilon^0=Q^0=\mu$.  As per Lemma \ref{thm technical gen}, we know that $\tilde{\alpha}_\epsilon^\mu$ is convex. This implies
$$\tilde{\alpha}_\epsilon^\mu(\tilde{Q}_\epsilon)\leq \int_\C \tilde{\alpha}_\epsilon^\mu\left(\W\circ(\bar{\omega} + \sqrt{\epsilon}W)^{-1} \right )Q(d\bar{\omega}) = \int_\C \int_0^1 g(t,\dot{\bar{\omega}}(t))dt\,Q(d\bar{\omega}) = \alpha^\mu_0(Q),$$
so taking limsup we conclude.

We proceed to show $(i)$. We take $Q_\epsilon\to Q$ and assume without loss of generality that $\tilde{\alpha}^\mu_\epsilon(Q_\epsilon) <\infty$. By the duality formula in Lemma \ref{thm technical gen}, and by Proposition \ref{th:Schilder-initial}, we have for any $F\in C_b(\C)$:
\begin{align*}
\liminf\limits_{\epsilon\downarrow 0} \tilde{\alpha}^\mu_\epsilon(Q_\epsilon) &\geq \liminf\limits_{\epsilon\downarrow 0} \left\{ \E^ {Q_\epsilon}[F] - \tilde\rho_\epsilon^\mu(F) \right\} \\ &= \E^Q[F] - \int_{\R^d}\,\, \sup_{\omega \in \C_x}\left(F(\omega) - \int_0^1g(t,\dot{\omega}(t))dt\right)\mu(dx),
\end{align*}
where we recall the notation $\C_x := \{\omega \in \C : \omega(0)=x\}$.
Now, the function 
$$\R^d \ni x\mapsto  \inf_{\omega \in \C_x}\left( \int_0^1g(t,\dot{\omega}(t))dt - F(\omega) \right),$$ is the pointwise supremum of all functions $h$ satisfying $h(x)+F(\omega)\leq \int_0^1g(t,\dot{\omega}(t))dt + \Psi_{\omega(0)}(x) $ for all $x \in \R^d$ and all $\omega \in \C$, where we define $\Psi_a(x) = +\infty$ if $x \neq a$ and $\Psi_a(x)=0$ otherwise. Hence we have
$$\liminf\limits_{\epsilon\downarrow 0} \tilde{\alpha}^\mu_\epsilon(Q_\epsilon) \geq \!\sup_{\substack{F\in C_b(\C)\\ h\in L^{1}(\R^d,\mu)}}\!\!\!\left\{ \E^Q[F] + \int hd\mu\,:\, h(x)+F(\omega)\leq \int_0^1g(t,\dot{\omega}(t))dt + \Psi_{\omega(0)}(x), \, \forall x,\omega \right \}.$$
By Kantorovich duality \cite[Theorem 1.3]{villani2003topics}, the right-hand side is equal to 
$$\inf_\pi\int_{\C\times \R^d}\left ( \int_0^1 g(t,\dot{\omega}(t))dt + \Psi_{\omega(0)}(x)   \right ) \pi(d\omega,dx),$$
where the infimum is over all $\pi \in \P(\C\times \R^d)$ with first marginal $Q$ and second marginal $\mu$.
Unless $\mu = Q^0$, this quantity is clearly infinite, and it is then straightforward to check that the entire expression reduces to  $\alpha^\mu_0(Q)$.
\end{proof}

As a final preparation for the proof of Corollary \ref{coro Schroedinger}, we need the following compactness lemma.

\begin{lemma}\label{lem tightness mikami}
The family $\{ \tilde{\alpha}^\mu_\epsilon :\epsilon\leq 1 \}$ is equicoercive, namely:
$$\bigcup_{\epsilon\leq 1}\{\tilde{\alpha}^\mu_\epsilon \leq c\}\text{ is tight for each $c\in\R$}.$$
\end{lemma}

\begin{proof}
This is the same argument as in the \emph{inf-tightness} part of the proof of Lemma \ref{thm technical}, which we provide in Appendix \ref{sec pending proofs} below. The point is that the initial distribution of the canonical process is independent of $\epsilon$, its quadratic variation is uniformly bounded in $\epsilon$, and its drift is bounded in $L^1$ independently of $\epsilon$ thanks to Assumption (TI) and the conditions $\tilde{\alpha}^\mu_\epsilon \leq c$.
\end{proof}

\noindent\textbf{Proof of Corollary \ref{coro Schroedinger}.}
With the notation we have built up, equality \eqref{eqMikamiLeo1} is equivalent to
\begin{align*}
\lim_{\epsilon\downarrow 0}\,\,\, \inf \left\{ \tilde{\alpha}^\mu_\epsilon(Q) \, :\,Q\in\mathcal{P}(\C), \ H(Q)=\nu_\epsilon \right \} & = \inf \left\{ \alpha^\mu_0(Q)\, :\,Q\in\mathcal{P}(\C), \ H(Q)=\nu  \right\}.
\end{align*}
We begin by proving the upper bound,
\begin{align*}
\limsup_{\epsilon\downarrow 0}\,\,\, \inf \left\{ \tilde{\alpha}^\mu_\epsilon(Q) \, :\,Q\in\mathcal{P}(\C), \ Q\circ H^{-1}=\nu_\epsilon \right \} & \leq\inf \left\{ \alpha^\mu_0(Q)\, :\,Q\in\mathcal{P}(\C), \ Q\circ H^{-1}=\nu  \right\}.
\end{align*}
If there is no $Q \in \P(\C)$ with $Q\circ H^{-1}=\nu$ the right-hand side is $+\infty$. Otherwise, for each $Q \in \P(\C)$ with $Q\circ H^{-1}=\nu$ we introduce $\tilde{Q}_\epsilon := Q * \W_\epsilon$ as in Lemma \ref{lem convolution}. By linearity of $H$ we have
\[
\tilde{Q}_\epsilon\circ H^{-1}=(Q\circ H^{-1})*(\W_\epsilon\circ H^{-1})=\nu_\epsilon.
\]
By Lemma \ref{lem convolution}, for each $Q \in \P(\C)$ we have
\[
\limsup_{\epsilon\downarrow 0}\,\,\, \inf \left\{ \tilde{\alpha}^\mu_\epsilon(Q') \, :\,Q'\in\mathcal{P}(\C), \ Q'\circ H^{-1}=\nu_\epsilon \right \} \leq\limsup_{\epsilon\downarrow 0} \tilde{\alpha}^\mu_\epsilon(\tilde Q_\epsilon) \leq \alpha^\mu_0(Q).
\]
Infimize over $Q \in \P(\C)$ satisfying $Q \circ H^{-1}=\nu$  to get the announced upper bound.

It remains to prove the lower bound,
\begin{align*}
\liminf_{\epsilon\downarrow 0}\,\,\, \inf \left\{ \tilde{\alpha}^\mu_\epsilon(Q) \, :\,Q\in\mathcal{P}(\C), \ Q\circ H^{-1}=\nu_\epsilon \right \} &\geq \inf \left\{ \alpha^\mu_0(Q)\, :\,Q\in\mathcal{P}(\C), \ Q\circ H^{-1}=\nu  \right\}.
\end{align*}
If the left-hand side is infinite there is nothing to prove. 
Otherwise, there exist sequences $\epsilon_n\downarrow 0$ and $Q_n \in \P(\C)$ with $Q_n \circ H^{-1} = \nu_{\epsilon_n}$ such that
\[
\lim_{n\rightarrow\infty}\tilde{\alpha}^\mu_{\epsilon_n}(Q_n) = \liminf_{\epsilon\downarrow 0}\,\,\, \inf \left\{ \tilde{\alpha}^\mu_\epsilon(Q) \, :\,Q\in\mathcal{P}(\C), \ Q\circ H^{-1}=\nu_\epsilon \right \}
\]
and also $\sup_n\tilde{\alpha}^\mu_{\epsilon_n}(Q_n) < \infty$. The latter property along with Lemma \ref{lem tightness mikami} ensures that we may pass to a further subsequence and assume that $Q_n \rightarrow Q$ for some $Q \in \P(\C)$. Continuity of $H$ implies $Q \circ H^{-1} = \lim_nQ_n \circ H^{-1} = \lim_n \nu_{\epsilon_n} = \nu$. Moreover, by Lemma \ref{lem convolution}, we have $\liminf_{n\rightarrow\infty}\tilde{\alpha}^\mu_{\epsilon_n}(Q_n) \ge \alpha^\mu_0(Q)$, and we deduce the aforementioned lower bound.

That the problems in \eqref{eq prob lhs eps} admit an optimizer, provided there exists a feasible element, follows from the compactness of the sub-level sets of $\tilde{\alpha}_\epsilon^\mu$ (see Lemma \ref{thm technical gen}), since the constraint $Q\circ H^{-1}=\nu_\epsilon$ is closed under weak convergence of measures. The analogous result for \eqref{eq prob rhs 0} follows taking $\epsilon=0$.

If an optimizer for $Q_\epsilon$ exists for all $\epsilon>0$, and if $\bar{Q}$ is an accumulation point of $\{Q_\epsilon\}_\epsilon$, then $\bar{Q}$ must be feasible for \eqref{eq prob rhs 0}. Thus there exists $Q$ an optimizer for \eqref{eq prob rhs 0}, or equivalently for $$\inf \left\{ \alpha^\mu_0(Q)\, :\,Q\in\mathcal{P}(\C)\,\text{with } Q\circ H^{-1}=\nu  \right\}.$$
Defining $\tilde Q_\epsilon$ as in Lemma \ref{lem convolution},we have
$$\alpha^\mu(Q)=\lim \tilde{\alpha}^\mu_\epsilon(\tilde Q_\epsilon) \geq \liminf \tilde{\alpha}^\mu_\epsilon(Q_\epsilon)\geq \alpha^\mu_0(\bar Q), $$
by Lemma \ref{lem convolution}. So $\bar Q$ is optimal for \eqref{eq prob rhs 0} as desired.\hfill\qedsymbol

{\ } \\

We now proceed to the proof of Corollary \ref{coro g sub quadratic}. From here on, we take $$Z=\mathbb R^d\text{ and }H(\omega)=\omega(1),$$ so we are in the classical situation. We will make use of a technical estimate for Brownian bridges. We denote by $$\W_{\epsilon}^{x,y}[a,b] \in \mathcal P(C([a,b];\mathbb R^d)),$$ the Brownian bridge from ``$x$ at time $a$ to $y$ at time $b$'' with instantaneous variance $\epsilon$. This is the law, on the space of continuous functions on $[a,b]$, of Brownian motion with volatility $\epsilon$ conditioned to start in $x$ and end in $y$. We refer to \cite[Theorem 40.3]{RoWi00} for a characterization of (multidimensional) Brownian bridges.

\begin{lemma}\label{lem technical bridge}
Let $a<b$. The canonical process admits under $\W_{\epsilon}^{x,y}[a,b]$ the decomposition $$W(t)=x + \int_a^t \frac{y - W(t)}{b-s} ds + \sqrt{\epsilon}B(t),$$
where $B$ is a standard $d$-dimensional Brownian motion on $[a,b]$. For all $1<r<2$ we have
\begin{align}\label{eq gaussian estimates}
\E^{\W_{\epsilon}^{x,y}[a,b]}\left[\int_a^b \left| \frac{y - W(t)}{b-s} \right|^r dt\right] \leq K_r|y-x|^r|b-a|^{1-r} + K_r |b-a|^{1-(r/2)}\epsilon^{r/2},
\end{align}
where $K_r < \infty$ is a constant depending only on $r$.
\end{lemma}
\begin{proof}
The claimed decomposition is classical \cite[Theorem 40.3]{RoWi00}.
To prove \eqref{eq gaussian estimates}, it suffices to consider the interval $[0,b-a]$ rather than $[a,b]$. Let $\delta=b-a$.
By conditioning of Gaussian distributions, we know that $W(t)$ is Gaussian with mean $\frac{\delta-t}{\delta}x+\frac{t}{\delta}y$ and variance matrix $\frac{\epsilon}{\delta} t(\delta-t)\text{Id}$, for each $t \in (0,\delta)$, under $\W_{\epsilon,\delta}^{x,y}$. From this, denoting $\W_1^1=\mathcal{N}(0,\text{Id})$, it follows
{
\begin{align*}
\E^{\W_{\epsilon,\delta}^{x,y}}\left[\int_0^\delta \left| \frac{y - W(t)}{b-s} \right|^r dt\right] & = \int_0^\delta \int_{\R^d} \left | \frac{y-x}{\delta}+z\,\sqrt{\frac{\epsilon t}{\delta(\delta-t)}} \right |^r d\W_1^1(z) dt\\
&\leq 2^{r-1}|y-x|^r\delta^{1-r} + 2^{r-1}\int_{\R^d}|z|^rd\W^1_1(z) \int_0^\delta \left ( \frac{\epsilon t}{\delta(\delta-t)}\right )^{r/2} dt  \\
&\le  K_r|y-x|^r\delta^{1-r} + K_r\delta^{1-(r/2)}\epsilon^{r/2}, 
\end{align*}}
where $K_r = 2^{r-1}\int_{\R^d}|z|^rd\W^1_1(z)\int_0^1\left ( \frac{t}{1-t}\right )^{r/2} dt$. Note that $K_r < \infty$ for $1 < r < 2$.
\end{proof}

\noindent\textbf{Proof of Corollary \ref{coro g sub quadratic}.}
Because of Lemma \ref{lem convolution}(i), the lower bound can be established exactly as in the proof of Corollary \ref{coro Schroedinger}. The delicate point is proving the upper bound
\begin{align}
\limsup_{\epsilon\downarrow 0}\,\,\, \inf \left\{ \tilde{\alpha}^\mu_\epsilon(Q) \, :\,Q\in\mathcal{P}(\C), \ Q^1=\nu \right \} & \leq\inf \left\{ \alpha^\mu_0(Q)\, :\,Q\in\mathcal{P}(\C), \ Q^1=\nu  \right\}, \label{pf:Schro-bridge1}
\end{align}
for which we cannot rely on Lemma \ref{lem convolution}(ii) as we did in the proof of Corollary \ref{coro Schroedinger}, because we are working now with $\nu$ instead of $\nu_\epsilon$ on the left-hand side.
If the right-hand side is infinite there is nothing to prove. Let us take any $Q$ with  ${\alpha}^\mu_0(Q)<\infty$ and $Q^1=\nu$. We introduce the measures
\[
\pi_{s,t}= Q\circ (W(s),W(t))^{-1}, \quad\quad \text{ and } \quad\quad \pi_{\epsilon,(s,t)}:= \pi_{s,t}*(\W_{\epsilon}^s\otimes\delta_0).
\]
That is, $\pi_{\epsilon,(s,t)} \in \P(\R^d \times \R^d)$ is the joint law of $(X(s)+\sqrt{s\epsilon}Z\, ,\,X(t))$, where $X\sim Q$ and $Z$ is an independent standard $d$-dimensional Gaussian.
The goal is to define now $\tilde Q_\epsilon$ satisfying the statement in Lemma \ref{lem convolution}(ii), but with $\tilde Q_\epsilon^1=\nu$ (and of course $\tilde Q_\epsilon^0=\mu$).

Let $\delta < 1$, which we will later send to zero. We will define first $\tilde Q_{\epsilon,\delta}$ by convolution of $Q$ and $\W_\epsilon$ in the time interval $[0,1-\delta]$, and we then steer toward the appropriate marginal $\nu$ at time $1$ by using a suitable mixture of Brownian bridges. Concretely, we define $\tilde Q_{\epsilon,\delta}$ uniquely by the four properties:
\begin{enumerate}
\item $\tilde Q_{\epsilon,\delta}\circ \left(\{W(t)\}_{t\leq 1-\delta}\right)^{-1}= (Q*P_\epsilon)\circ \left(\{W(t)\}_{t\leq 1-\delta}\right)^{-1}$
\item $\tilde Q_{\epsilon,\delta}\circ (W(1-\delta),W(1))^{-1} = \pi_{\epsilon,(1-\delta,1)} $
\item $\tilde Q_{\epsilon,\delta}(W(1)\in \cdot \, | \, \{W(t)\}_{t\leq 1-\delta}) = \tilde Q_{\epsilon,\delta}(W(1)\in \cdot \, | \, W(1-\delta))$, a.s.
\item $\tilde Q_{\epsilon,\delta}(\{W(t)\}_{t\in [1-\delta,1]} \, | \, W(1), \ \{W(t)\}_{t \le 1-\delta}) = P_\epsilon^{W(1-\delta),W(1)}[1-\delta,1]$, a.s.
\end{enumerate} 
We remark that $\tilde  Q_{\epsilon,\delta}$ is a semimartingale law for which the martingale part is $\sqrt{\epsilon}$ times a Brownian motion and, crucially, for which the time-$0$ and time-$1$ marginals are, respectively,
\[
\tilde Q_{\epsilon,\delta}^0=\mu \quad \text{ and } \quad \tilde Q_{\epsilon,\delta}^1= \nu.
\]
Because $Q_{\epsilon,\delta}=Q * \W_\epsilon$ on $\F_t$, we also have
\begin{align}
\tilde{\alpha}^\mu_\epsilon(\tilde Q_{\epsilon,\delta})= \E^{Q*\W_\epsilon}\left [ \int_0^{1-\delta} g\left (q^{Q*\W_\epsilon}(t)\right )dt \right ] +  A_{\epsilon, [1-\delta,1]}^\mu, \label{pf:Schro-alpha}
\end{align}
where (recalling the semimartingale decomposition of $\W_{\epsilon}^{x,y}[1-\delta,1]$ stated in Lemma \ref{lem technical bridge})
$$A_{\epsilon, [1-\delta,1]}^\mu:= \int_{\R^d \times \R^d} \E^{\W_{\epsilon}^{x,y}[1-\delta,1]}\left [\int_{1-\delta}^1 g\left(\frac{y - W(t)}{1-t}\right)dt \right ] \pi_{\epsilon,(1-\delta,1)}(dx,dy).$$
Recall now the assumption that $g(q)$ grows like $|q|^r$, and use \eqref{eq gaussian estimates}  to bound $A_{\epsilon, [1-\delta,1]}^\mu$:
$$A_{\epsilon, [1-\delta,1]}^\mu \leq C \delta^{1-r}\int_{\R^d \times \R^d} |x-y|^r  \pi_{\epsilon,(1-\delta,1)}(dx,dy) + C\epsilon^{r/2}\delta^{1-(r/2)},$$
where $C < \infty$ is a constant depending only on $r$ and $g$.
Using the definition of $\pi_{\epsilon,(1-\delta,1)}$ and Jensen's inequality we deduce (keeping the same constant $C$)
\begin{align*}
A_{\epsilon, [1-\delta,1]}^\mu &\leq 2^{r-1}C\delta \E^Q\left[\left |\frac{W(1)-W(1-\delta)}{\delta}\right|^r\right]+ 2^{r-1}C\delta^{1-r}((1-\delta)\epsilon)^{r/2} \int_{\R^d} |z|^rd\W_1^1(z) \\
	&\quad + C\epsilon^{r/2}\delta^{1-(r/2)} \\
&\leq 2^{r-1} C\delta \E^Q\left[\frac{1}{\delta}\int_{1-\delta}^1|\dot W(t)|^r dt\right]+ 2^{r-1}C\delta^{1-r}\epsilon^{r/2} \int_{\R^d} |z|^rd\W_1^1(z) + C\epsilon^{r/2}\delta^{1-(r/2)}.
\end{align*}
The first term on the right-hand side goes to zero as $\delta \downarrow 0$ since we assumed $\alpha^\mu_0(Q)=\E^Q\left[\int_{0}^1|\dot W(t)|^r dt\right]$ to be finite. If we set $\delta = \sqrt{\epsilon}$ then the second and third terms vanish as well, as $\epsilon \downarrow 0$.
Let us finally define $\tilde Q_\epsilon := \tilde Q_{\epsilon,\sqrt{\epsilon}}$. By dominated convergence $\tilde Q_\epsilon \to Q$, while on the other hand we have seen that $A_{\epsilon, [1-\sqrt{\epsilon},1]}^\mu \to 0$. Recalling the equation \eqref{pf:Schro-alpha}, the proof of \eqref{pf:Schro-bridge1} would be concluded if we can show that
$$\limsup_{\epsilon\to 0}\E^{Q*\W_{\epsilon}}\left [ \int_0^{1-\sqrt{\epsilon}} g\left (q^{Q*\W_{\epsilon}}(t)\right )dt \right ] \leq \alpha^\mu_0(Q). $$
Let us call $(X,Y)$ the canonical process on $\C_0\times \C_0$ equipped with the reference measure $Q\otimes P_\epsilon$. Of course $Q*\W_{\epsilon}=Q\otimes P_\epsilon\circ (X+Y)^{-1}$ and $X$ has absolutely continuous trajectories. It follows that for every $t \in [0,1]$:
\[
q^{Q*\W_{\epsilon}}(t,\omega)=\E^{Q\otimes \W_\epsilon}[\dot X(t)\,|\, \{X(s)+Y(s)\}_{s\leq t}=\{\omega(s)\}_{s\leq t}], \quad Q*\W_\epsilon-a.e. \ \omega.
\]
By Jensen's inequality, we conclude
$$\limsup_{\epsilon\to 0}\E^{Q*\W_{\epsilon}}\left [ \int_0^{1-\sqrt{\epsilon}} g\left (q^{Q*\W_{\epsilon}}(t)\right )dt \right ] \leq \limsup_{\epsilon\to 0}\E^{Q\otimes\W_{\epsilon}}\left [ \int_0^{1-\sqrt{\epsilon}} g\left (\dot X(t)\right )dt \right ] =
 \alpha^\mu_0(Q). $$
{\ } \hfill\qedsymbol
}

\appendix

\section{Proofs of properties of $\alpha^g$}
\label{sec pending proofs}

We collect here the belated proofs of some technical results.

\begin{lemma} \label{le:H1-tightness}
Suppose $q_n \in \L$ and $A_n(t) = \int_0^tq_n(s)ds$. Suppose there exists $a > 0$ such that, for each $n$,
\begin{align}
\E\int_0^tg(t,q_n(t))dt \le a. \label{def:H1-tightness-bound}
\end{align}
Then there exist a continuous process $A$, a subsequence $A_{n_k}$ which converges in law in $\C$ to $A$, and a process $q \in \L$ such that
\begin{align}
\E\int_0^tg(t,q(t))dt \le \liminf_{k\rightarrow\infty}\E\int_0^tg(t,q_{n_k}(t))dt \label{def:H1-tightness-liminf}
\end{align}
and $A(t) = \int_0^tq(s)ds$. In particular, $(A_n)$ is tight.
\end{lemma}
\begin{proof}
We first check tightness.
By Assumption (TI), for each $c > 0$ we may find $N > 0$ such that $g(t,q) \ge c|q|$ whenever $|q| \ge N$. Moreover, there exists $b\ge0$ such that $g(t,q)\ge -b$ for all $(t,q)$. In particular, for all $(t,q)$ we have $|q| \le N + \frac 1c(g(t,q) + b)$. 
Hence, for $0 \le s < t \le 1$,
\begin{align*}
|A_n(t)-A_n(s)| &\le \int_s^t|q_n(u)|du \le \frac{1}{c}\int_s^t (g(u,q(u)) + b) \,du + N(t-s) \\
	&\le \frac{1}{c}\int_0^1 g(u,q(u)) \,du + \frac{b}{c} + N(t-s).
\end{align*}
Hence, for any $\delta_n \downarrow 0$, \eqref{def:H1-tightness-bound} yields
\begin{align*}
\limsup_{n\rightarrow\infty} \sup_n\sup_\tau \E|A_n(\tau+\delta_n) - A_n(\tau)| &\le \limsup_{n\rightarrow\infty}\left(\frac{a+b}{c} + N\delta_n \right) = \frac{a+b}{c},
\end{align*}
where the $\sup_\tau$ is over all stopping times with values in $[0,1-\delta_n]$.
As $c > 0$ was arbitrary, this shows that 
\begin{align*}
\lim_{n\rightarrow\infty} \sup_\tau \sup_n\E|A_n(\tau+\delta_n) - A_n(\tau)| = 0,
\end{align*}
and from Aldous' criterion for tightness  \cite[Theorem 16.11]{kallenberg} we conclude that $(A_n)$ is tight.

Passing to a subsequence and applying Skorohod's representation, let us now assume that there exists a continuous process $A$ such that $A_n \rightarrow A$ almost surely in $\C$, with all processes defined on some common probability space $(\Omega,\F,\PP)$. From \eqref{def:H1-tightness-bound}, assumption (TI), and the criterion of de la Vall\'ee Poisson, we conclude that $\{q^n : n \in \N\} \subset L^1 := L^1([0,1] \times \Omega, \, dt \otimes d\PP)$ is uniformly integrable and thus weakly precompact. By passing to a further subsequence, we may now assume that $q^n \rightarrow q$ weakly in $L^1$. Because $g$ is bounded from below and lower semicontinuous in its second variable, the map $q \mapsto \E\int_0^1 g(t,q(t))dt$ is lower semicontinuous in the norm topology of $L^1([0,1] \times \Omega)$ by Fatou's lemma. Because it is also convex, this map is therefore weakly lower semicontinuous on $L^1$. This yields \eqref{def:H1-tightness-liminf}.
Lastly, by dominated convergence, it holds for each bounded random variable $Z$ that
\begin{align*}
\E[ZA(t)] = \lim_{n\rightarrow\infty}\E[ZA^n(t)] = \lim_{n\rightarrow\infty}\E\left[Z\int_0^tq^n(s)ds\right] = \E\left[Z\int_0^tq(s)ds\right].
\end{align*}
Hence $A(t) = \int_0^tq(s)ds$ a.s.\ for each $t$, and by continuity we have $A = \int_0^\cdot q(s)ds$ a.s. 
\end{proof}

{\ } \\

\noindent\textbf{Proof of Lemma \ref{thm technical}} {\ }

\textbf{Convexity}: Let $\lambda\in [0,1]$, and fix $Q_0,Q_1\in {\cal P}^*$. We work on an extended probability space $\C \times \{0,1\}$, and we write $(W,X)$ to denote the identity map on this space. We define a measure $M$ on $\C \times \{0,1\}$ by requiring that the second marginal of $M$ be $\lambda \delta_0 + (1-\lambda)\delta_1$, and the conditional law of $W$ given $X$ be $Q_X$. In particular, the first marginal of $M$ is precisely $Q:=\lambda Q_0 + (1-\lambda)Q_1$.
Abbreviate $q_i:=q^{Q_i}$.
It easily follows that the process
\[
W(t) - \int_0^t q_X(s)ds
\]
defines an $M$-Brownian motion with respect to the filtration $\overline\FF=(\overline\F_t)_{t \in [0,1]}$ defined by $\overline\F_t=\F_t \otimes \sigma(X)$ on the product space. Now define the process $q=(q(t))_{t \in [0,1]}$ on $\C \times \{0,1\}$ to be the optional projection of the process $(q_X(t))_{t \in [0,1]}$ on the filtration generated by $W$. In particular,
\[
q(t) = \E^M[q_X(t) \, | \, (W_s)_{s \le t}] = \E^M[ {\bf 1}_{\{X=0\}}q_0(t) + {\bf 1}_{\{X=1\}}q_1(t) \, | \, (W_s)_{s \le t}  ].
\]
A quick computation reveals that $W - \int_0^\cdot q(t)dt$
is still an $M$-martingale. But this process is adapted to the filtration of $W$, so it may be viewed as a martingale on $(\C,\FF,Q)$, where we recall that $Q$ is the first marginal of $M$.
Using L\'evy's criterion, this process is then a Brownian motion on $(\C,\FF,Q)$. It follows that $Q \in {\cal P}^*$ and $q=q^Q$. Finally, using Jensen's inequality, we compute
\begin{align*}
\lambda\tilde{\alpha}^g(Q_0) + (1-\lambda)\tilde{\alpha}^g(Q_1) &= \lambda \E^{Q_0}\left[\int_0^1 g(t,q_0(t))dt\right]+(1-\lambda)\E^{Q_1}\left[\int_0^1 g(t,q_1(t))dt\right] \\ 
	&= \E^M\left[\int_0^1 g(t, q_X(t))dt\right] \\ 
	&\ge \E^M\left[\int_0^1 g(t, q(t))dt\right] = \E^{Q}\left[\int_0^1 g(t,q(t))dt\right] \\
	&= \tilde{\alpha}^g(Q).
\end{align*}

{
\textbf{Inf-compactness:} 
Let $a\in \R$ and $\Lambda_a := \{Q:\tilde{\alpha}^g(Q)\leq a\}$. It is convenient in this step and the next to define
\[
W^Q(t) := W(t) - \int_0^t q^Q(s)ds, \quad t \in [0,1],
\]
for $Q \in \P^*$, recalling that $W^Q$ is a $Q$-Brownian motion by definition of $\P^*$. Letting $A^Q(t) := \int_0^tq^Q(s)ds$, it follows from Lemma \ref{le:H1-tightness} that $\{Q \circ (A^Q)^{-1} : Q \in \Lambda_a\} \subset \P(\C)$ is tight. On the other hand, $\{Q \circ (W^Q)^{-1} : Q \in \Lambda_a\} = \{\W\}$ is a singleton and thus tight. Since each marginal is tight, we deduce that $\{Q \circ (W^Q,A^Q)^{-1} : Q \in \Lambda_a\} \subset \P(\C \times \C)$ is tight. Finally, by continuous mapping, the set $\{Q \circ (W^Q + A^Q)^{-1} : Q \in \Lambda_a\} = \Lambda_a$ is tight.
}

{
\textbf{Lower semicontinuity:} Suppose $\{Q_n : n \in \N \} \subset \Lambda_a$ with $Q_n \rightarrow Q$ weakly for some $Q \in \P(\C)$. We must show that $Q$ belongs to $\Lambda_a$. Define the continuous process
\[
A^n(t) = \int_0^t q^{Q_n}(s)ds = W(t) - W^{Q_n}(t),
\]
for each $n$. Since $Q_n \circ (W^{Q_n})^{-1}$ equals Wiener measure for each $n$, we conclude that $\{Q_n \circ (W,W^{Q_n})^{-1} : n \in \N \}$ is tight, and thus $\{Q_n \circ (W,W^{Q_n},A^n)^{-1} : n \in \N\}$ is tight. 
Relabeling a subsequence, suppose that $Q_n \circ (W,W^{Q_n},A^n)^{-1}$ converges weakly to the law of some $\C^3$-valued random variable $(X,B,A)$.
Using Lemma \ref{le:H1-tightness}, we may assume also that $A(t) = \int_0^tq(s)ds$ for some process $q$ satisfying
\[
\E\int_0^1g(t,q(t))dt \le \liminf \E\int_0^1g(t,q^{Q_n}(t))dt \le a.
\]
Clearly, the law of $B$ is Wiener measure. Moreover, $(W^{Q_n}(s)-W^{Q_n}(t))_{s \in [t,1]}$ is independent of $(W(s),W^{Q_n}(s),A^n(s))_{s \le t}$ for each $t \in [0,1]$, and thus $(B(s)-B(t))_{s \in [t,1]}$ is independent of $(X(s),B(s),A(s))_{s \le t}$. In particular, $B$ is a Brownian motion with respect to the filtration generated by $X$, $B$, and $q$. Finally, notice that
\[
X(t) = B(t) + A(t) = B(t) + \int_0^tq(s)ds,
\]
as the same relation holds in the pre-limit.
A standard argument (see \cite[Exercise (5.15)]{Revuz1999}) shows that $X - \int_0^{\cdot}\widehat{q}(s)ds$ is a Brownian motion, where $\widehat{q}$ is the optional projection of $q$ onto the filtration generated by $X$. By convexity of $g(t,\cdot)$, we have
\[
\E\int_0^1g(t,\widehat{q}(t))dt \le \E\int_0^1g(t,q(t))dt \le a.
\]
Recalling that $Q$ denoted the law of $X$, we conclude that $Q \in \mathcal{P}^*$ and thus $Q \in \Lambda_a$.
}

\textbf{Reverse conjugacy:} By definition $$\tilde{\alpha}^g(Q)\geq\sup_{F\in B_b(\C)}\{\E^Q[F]-\tilde{\rho}^g(F)\}\geq\sup_{F\in C_b(\C)}\{\E^Q[F]-\tilde{\rho}^g(F)\}.$$
Recalling the previous results showing convexity and lower semicontinuity of $\tilde{\alpha}^g$, we may apply the Fenchel-Moreau theorem with respect to the dual pairing between $C_b(\C)$ and the space of measures on $\C$ to get equality above.\hfill\qedsymbol

We close by elaborating slightly on the dual representation of BSDE supersolutions, which was discussed to some extent on page \pageref{eq rho = Y}. In particular, the following slight adaptation of results of \cite{tarpodual} was used in Lemma \ref{le:BSDE-timedep}, which extended equation \eqref{eq rho = Y} to nonzero times $t$.

\begin{lemma}
\label{lem:rep Qt}
	Let $F \in C_b(\C)$.
	The minimal supersolution of the BSDE
	\begin{equation*}
		dY(t) = - g^*(t, Z(t))\,dt + Z(t)\,dW_t, \quad Y(1) = F
	\end{equation*}
	admits the representation 
	\begin{equation*}
		Y(t) = \esssup_{Q \in \Q_t} \E^Q\left[F(W) - \int_t^1g(u,q^Q(u))\,du \, \Big| \, \F_t   \right] \quad P\text{-a.s. for all } t\in [0,1],
	\end{equation*}
	where $\Q_t$ is the set of $Q \in \Q$ such that $Q=\W$ on $\F_t$.
\end{lemma}
\begin{proof}
	Since $\Q_t \subseteq \Q$, ''$\ge$'' follows by \cite[Theorem 3.4]{tarpodual}.
	Reciprocally, since by (the first part of the proof of) \cite[Proposition 4.2]{tarpodual} the set $\left\{ \E_Q\left[F(W) - \int_t^1g(u,q^Q(u))\,du \, \Big| \, \F_t   \right] : Q \in \Q\right\}$ of random variables is directed, it holds
	\begin{equation*}
		Y_t = \lim_{n\to \infty}\E^{Q^n}\left[F(W) - \int_t^1g(u,q^{Q^n}(u))\,du \, \Big| \, \F_t   \right]
	\end{equation*}
	for a sequence $Q^n \in \Q$.
	Put $q^n(u) := q^{Q^n}(u)1_{[t,1]}(u)$ and let $\bar Q^n$ be such that $q^{\bar Q^n} =  q^n$.
	Then, $\bar Q^n \in \Q_t$ and it follows from Bayes' rule that
	\begin{align*}
		Y(t)&=\lim_{n\to \infty} \E\left[e^{\int_t^1q^{Q^n}(u)\,dW(u) - \frac12\int_t^1|q^{Q^n}(u)|^2du}\left(F(W) - \int_t^1g(u,q^{Q^n}(u))\,du\right) \, \Big| \, \F_t   \right]\\
		& = \lim_{n\to \infty} \E\left[e^{\int_t^1q^{n}(u)\,dW(u) - \frac12\int_t^1|q^{n}(u)|^2du}\left(F(W) - \int_t^1g(u,q^n(u))\,du\right) \, \Big| \, \F_t   \right]\\
		&= \lim_{n\to \infty} \E^{\bar Q^n}\left[F(W) - \int_t^1g(u,q^{\bar Q^n}(u))\,du \, \Big| \,\F_t   \right],
	\end{align*}
	which proves "$\le$".
\end{proof}

\bibliographystyle{abbrvnat}

\bibliography{references-Funct-ineq-conv}

\end{document}